\documentclass[10pt,reqno]{amsart}

\usepackage{amssymb}
\usepackage{amsfonts}
\usepackage{amsthm}
\usepackage{amsmath}
\usepackage{bbm}
\usepackage{xcolor}
\usepackage{enumerate}
\usepackage[title]{appendix}
\usepackage[normalem]{ulem}

\numberwithin{equation}{section}
\allowdisplaybreaks

\newtheorem{theorem}{Theorem}[section]
\newtheorem{proposition}[theorem]{Proposition}

\newtheorem{lemma}[theorem]{Lemma}
\newtheorem{corollary}[theorem]{Corollary}

\theoremstyle{definition}

\theoremstyle{remark}

\newcommand{\R}{\mathbb{R}}
\newcommand{\Rd}{\mathbb{R}^d}

\newcommand{\N}{\mathbb{N}}
\newcommand{\Z}{\mathbb{Z}}

\DeclareRobustCommand{\rchi}{{\mathpalette\irchi\relax}}
\newcommand{\irchi}[2]{\raisebox{\depth}{$#1\chi$}} 

\renewcommand{\hat}{\widehat}
\newcommand{\eps}{\varepsilon}

\newcommand{\scriptA}{\mathcal{A}}

\newcommand{\scriptE}{\mathcal{E}}

\newcommand{\scriptI}{\mathcal{I}}
\newcommand{\scriptJ}{\mathcal{J}}

\newcommand{\scriptR}{\mathcal{R}}
\newcommand{\scriptS}{\mathcal{S}}

\newcommand{\jp}[1]{\langle{#1}\rangle}
\newcommand{\qtq}[1]{\quad\text{#1}\quad}

\DeclareMathOperator*{\supp}{supp}

\DeclareMathOperator*{\dist}{dist}
\DeclareMathOperator*{\wklim}{wk-lim}

\begin{document}
\title{Sharp Fourier restriction to monomial curves}

\author{Chandan Biswas and Betsy Stovall}
\address{Department of Mathematics, Indian Institute of Technology Bombay, Mumbai 400076}
\email{cbiswas@iitb.ac.in}
\address{Department of Mathematics, University of Wisconsin, Madison, WI 53706}
\email{stovall@math.wisc.edu}

\date\today
\maketitle
\begin{abstract}
We establish  lower bounds for the operator norms of the Fourier restriction/extension operators associated to monomial curves with affine arclength measure. Furthermore, we prove that the set of all extremizing sequences of such an operator is precompact modulo the operator's symmetry group if and only if the operator norm is strictly larger than this threshold. For the proof, we introduce a number of new ingredients, some of which may be applicable to analogous questions on more general manifolds.
\end{abstract}



\section{Introduction}

In this article, we take a step toward the development of  profile decomposition techniques for $L^p \to L^q$ Fourier extension operators associated to general polynomial curves in $\R^d$.  We treat the case of monomial curves,  extending techniques that had previously only been applied to specific manifolds to a larger class.  We focus here on the case of monomial curves, but we develop several new techniques, which will be applied to the general polynomial case in a sequel.

We begin with a brief historical review of the Fourier restriction problem for polynomial curves.  Let $\gamma :\R \to \Rd$ be a $d$-times continuously differentiable parametrization of a curve. The affine arclength measure along the image of $\gamma$ equals the pushforward by $\gamma$ of the measure $\lambda_\gamma\, dt:= |L_\gamma |^{\frac{2}{d^2 + d}}\, dt$, where $L_\gamma:= \text{det} (\gamma' , \ldots, \gamma^{(d)} )$, referred to here (incorrectly, but conveniently) as the torsion.  There is by now a large body of literature establishing $L^p(\R;\lambda_\gamma\,dt) \to L^q(\R^d)$  bounds for the Fourier extension operator 
$$
\scriptE_\gamma f(x) := \int_\R e^{i x \cdot \gamma (t)}f(t) \lambda_\gamma(t)\, dt, \quad x \in \Rd,
$$
initially defined on (say) smooth, compactly supported $f$, that are uniform over all curves in certain large classes.  For example, the following result is known for the polynomial class.  

\begin{theorem}[\cite{Dendrinos_Wright_2010,Sjolin_74, Stovall16}]\label{T:extn polynomial} 
Let $d \geq 2$, $1 \leq p < \tfrac{d^2+d+2}2$, and set $q:=\tfrac{d^2+d}2 p'$. For each $N$, there exists a constant $C_{d,N,p}$ such that for every polynomial $\gamma:\R \to \R^d$ of degree at most $N$, the corresponding extension operator $\scriptE_\gamma$ extends as a bounded linear operator from $L^p(\R;\lambda_\gamma\,dt)$ to $L^q(\R^d)$, obeying
$$
\|\scriptE_\gamma f\|_{L^q(\R^d)} \leq C_{d,N,p} \|f\|_{L^p(\R;\lambda_\gamma\,dt)}.
$$
\end{theorem}

Though the affine arclength measure is parametrization independent, the current proofs establish bounds that depend on (say) the polynomial degree of the parametrization ($C_{d, N, p} \to \infty$ in Theorem~\ref{T:extn polynomial} as $N \to \infty$). It is thus of interest to identify the geometric features that may drive up the operator norm as one expands the class of curves under consideration (say, by allowing the polynomial degree to increase), and additionally, to determine whether maximizers exist. We address these questions by examining the behaviour of extremizing sequences (those that saturate the operator norm), obtaining lower bounds for the $L^p \to L^q$ operator norms of the Fourier restriction operator for all values of $p$, and implications of such bounds for the question of existence of extremizers and convergence of extremizing sequences. These lower bounds depend only on the parity of the monomial degrees, and not on their magnitude.

A first step toward such a result was taken in \cite{BiswasStovall20}, in which the authors developed concentration compactness methods for $L^p \to L^q$ Fourier restriction/extension associated to the moment curve. In this article, we take the natural next step by extending those results to monomial curves, for which it was proved by Drury--Marshall \cite{drury1987fourier} and Bak--Oberlin--Seeger \cite{BOS2009-AJM} that $\scriptE_\gamma$ extends as a bounded linear operator from $L^p(\R;\lambda_\gamma\,dt) \to L^q(\R^d)$ for $(p,q)$ in the same range as for the moment curve, that is, for $q = \tfrac{d^2 + d}2 p' > \tfrac{d^2 + d + 2}2$.  Moreover, Bak--Oberlin--Seeger \cite{BOS2009-AJM, BOS2013-Crelle} established degree-independent upper bounds for the operator norms. 

To obtain our result for monomial curves, we introduce a number of new ingredients. Most notably, we introduce a framework for analyzing a sequence that blows up simultaneously at two ``antipodal'' points on a general monomial curve. By contrast, earlier results are mostly limited to specific model hypersurfaces, and largely address the $p=2$ case.  (Examples include the cone \cite{BahouriGerard99, KenigMerle08, Negro22}, the paraboloid and higher-order variants  \cite{BegoutVargas07, BrocchiSilvaRene20, CarlesKeraani07, Foschi15, HundertmarkZharnitsk06,  KenigMerle06, KillipTaoVisan08, KillipVisanZhang07, Kunze03, MerleVega98, QuilodranSilva16, JiangPausaderShao10, JiangShaoStovall14}, the cubic/KdV curve in $\R^2$ \cite{FarahVersieux18, FrankSabin18, KenigPonceVega2000,  Pigott16, Shao09}, elliptic hyperboloid \cite{KillipStovallVisan12}, and sphere~\cite{CarneiroFoschiSilvaThiele17, FlockStovall22, Quilodran22, Shao16}.  See   \cite{FoschiSilva17} or \cite{NegroSilvaThielesurvey2022} for a more comprehensive discussion of the recent literature on hypersurfaces.) In developing these results, one challenge posed by monomials and more general polynomial curves, as opposed to hypersurfaces, is the anisotropic decay of the Fourier transform of the associated measures. Thus, in the case of curves and other higher co-dimension manifolds, many arguments that are robust for (say) elliptic hypersurfaces are sensitive to perturbations of the underlying manifold. For example, a rescaled concentrating sequence may not have a fixed dominating function that lies in the $L^q$ space of interest.

For concreteness, we consider $\gamma$ of the form
\begin{equation}\label{assumptions on gamma}
    \begin{gathered}
\gamma(t) = \bigl(\tfrac{t^{l_1}}{l_1!}, \ldots, \tfrac{t^{l_d}}{l_d!}\bigr), \qquad 1 \leq l_1 < \cdots < l_d, 
\end{gathered}
\end{equation}
where $\vec l_\gamma := (l_1, \ldots, l_d) \in \N^d$ and $|\vec l_\gamma| :=\sum_{i = 1}^d l_i > \tfrac{d^2 + d}2$.  For $(p,q)$ in the range of boundedness, $q=\tfrac{d(d+1)}2 p' > p$, we denote the corresponding operator norm for the restriction/extension operators by
\begin{equation}\label{define B_gamma, p}
B_{\gamma, p} := \|\scriptE_\gamma\|_{L^p(\lambda_\gamma) \to L^q}.
\end{equation}
We recall that $B_{\gamma,p} = B_{A\gamma,p}$ for any invertible affine mapping $A$ of $\R^d$. We are interested in the question of when there exist nonzero functions $f$ such that $\|\scriptE_\gamma f \|_q= B_{\gamma, p}  \|f\|_{L^p (\lambda_\gamma)}$. We call such functions extremizers of $\scriptE_\gamma : L^p(\lambda_\gamma) \to L^q$. Additionally, we are interested in the behavior of normalized ($\|f_n\|_{L^p(\lambda_\gamma)} \equiv 1$) sequences $\{f_n\}$ that are extremizing in the sense that $\lim_n \|\scriptE_\gamma f_n\|_q  = B_{\gamma, p}$.

One obstruction to precompactness of such sequences is the existence of a noncompact symmetry group for $\scriptE_\gamma$.  By a symmetry of $\scriptE_\gamma:L^p(\lambda_\gamma) \to L^q$, we mean an isometry $S$ of $L^p(\lambda_\gamma)$ for which there exists an isometry $T_S$ of $L^q$ such that $\scriptE_\gamma \circ S = T_S \circ \scriptE_\gamma$.  Two symmetries in particular, scaling:
\begin{equation}\label{E : def scaling}
\scriptE_\gamma (|\delta|^{\frac{-2 |\vec l|}{(d^2 + d) p}} f(\tfrac{\cdot}{\delta})) (x) = |\delta|^{ \frac {|\vec l|}q} \scriptE_\gamma f(D_\delta^{\vec l} x), \quad D_\delta^{\vec l} x :=(\delta^{l_j}x_j)_{j=1}^d, 
\end{equation}
for $x \in \R^d$ and $\delta \neq 0$, where $\vec l:=\vec l_\gamma$, and modulation, 
$
\scriptE_\gamma(e^{i x_0 \cdot \gamma} f) (\cdot)= \scriptE_\gamma f (\cdot + x_0),
$
as well as the group they generate, are central to our analysis.  (We note that the dilations include time reversal, corresponding to dilation by $\delta = -1$.)  

In contrast with the moment curve,
$$
\gamma_0 (t) := (t, \tfrac{t^2}{2!}, \ldots, \tfrac{t^d}{d!}), \quad  t \in \R, \qquad \vec l_0:= (1,\ldots,d),
$$
higher order monomial curves lack a translation symmetry, and correspondingly lack a true dilation symmetry about points other than zero. 
However, sufficiently near $t_0 \neq 0$, we may approximate $\gamma$ by an affine copy of $\gamma_0$, leading to an approximate scaling symmetry:  
\begin{gather*}
\scriptE_\gamma(\delta^{\frac{-1}p}f(\tfrac{\cdot-t_0}\delta))(x)
=
\delta^{\frac1{p'}} e^{ix\cdot\gamma(t_0)} \scriptE_{\gamma_0+o(1)}((\lambda_\gamma(t_0)+o(1))f)(D_\delta \, T_\gamma (t_0)^t\, x),
\end{gather*}
where $T_\gamma (t_0)$ denotes the torsion matrix,
\begin{equation} \label{E:def T}
T_\gamma (t_0) := (\gamma'(t_0), \ldots, \gamma^{(d)}(t_0)), \qquad D_\delta : = D^{\vec{l_0}}_\delta,
\end{equation}
and the $o(1)$ terms tend to 0 locally uniformly as $\delta \searrow 0$.

By analyzing these approximate symmetries, we will deduce the lower bound for $B_{\gamma,p}$ given in the following theorem.  However, we first need a bit of additional notation.  For $1 < p, q < \infty$, let $\psi_{p, q} : [0, 1] \to \R$ denote the function
\begin{equation}\label{E : define Psi_p,q}
\psi_{p, q}(t) :=  {\frac1{2\pi} \int_0^{2\pi}\bigl|1+e^{i\theta}t\bigr|^q d\theta}{ \bigl(1+t^p)^{-q/p}}, \qquad \Psi_{p,q}:=\sup_{t \in [0,1]}\psi_{p,q}(t).
\end{equation}
By the dominated convergence theorem, $\psi_{p, q}$ is continuous, so $\Psi_{p,q}$ is a maximum.

\begin{theorem}[\textbf{Lower bound}]\label{T : conc energy lower bound}
Let $q = \frac {d^2 + d}2 p'> \frac {d^2 + d + 2}2$ and $\gamma : \R \to \Rd$ be a monomial. Then
\begin{equation} \label{E:Bconc}
B_{\gamma, p} 
\geq B_{\gamma,p}^{\rm{conc}}
:= B_{\gamma_0,p} \times 
\begin{cases}
\Psi_{p,q}^{1/q}, & \text{if each} \,\, l_i \,\,\text{is odd} 
\\
2^{1/p'}, & \text{if each $l_i$ is even}
\\
1, & \text{otherwise}.
\end{cases}
\end{equation}
Additionally, if $B_{\gamma,p} = B_{\gamma,p}^{\rm{conc}}$, and $f,g$ are extremizers for $\scriptE_{\gamma_0}:L^p \to L^q$ with $|\scriptE_{\gamma_0}f| = |\scriptE_{\gamma_0}g|$, a.e., then the following family of functions is extremizing for $\scriptE_\gamma:L^p(\lambda_\gamma) \to L^q$, as $\delta \searrow 0$\emph{:}
\begin{equation} \label{E:fdelta}
f_\delta(t):=
\begin{cases}
[\delta^{- \frac1p}f(\tfrac{t-1}\delta) + \alpha \delta^{-\frac1p}\overline{g(\tfrac{-(t+1)}\delta)}]\rchi_{\{|t| \leq 2\}}, \qquad & \text{if each} \,\, l_i \,\,\text{is odd,} 
\\
[\delta^{-\frac1p}f(\tfrac{t-1}\delta) + \delta^{-\frac1p} f(\tfrac{-(t+1)}\delta)] \rchi_{\{|t| \leq 2\}}, & \text{if each} \,\, l_i \,\,\text{is even,}\\
\delta^{- \frac1p}f(\tfrac{t-1}\delta) \rchi_{\{|t| \leq 2\}}, & \text{otherwise}.
\end{cases}
\end{equation}
Here $\alpha \in [0,1]$ is a maximum point for $\psi_{p,q}$.  
\end{theorem}

We recall from~\cite{BiswasStovall20} that there exists an extremizer of $\scriptE_{\gamma_0} : L^p \to L^q$ and note that sequences of the form $\{f_\delta\}_{\delta \searrow 0}$ are not precompact modulo symmetries of $\scriptE_\gamma$. Conversely, we will show that, after application of suitable symmetries, any extremizing sequence must contain a subsequence that either converges in norm or that concentrates at $1$ or $2$ points (depending on the parity of the entries of $\vec l_\gamma$). When $B_{\gamma,p} > B_{\gamma,p}^{\rm{conc}}$, we show that the latter case is not possible, leading to the following result.

\begin{theorem}[\textbf{Main result}]\label{T : existence}
Let $q = \frac {d^2 + d}2 p'> \frac {d^2 + d + 2}2$ and let $\gamma : \R \to \Rd$ be a monomial. Then for each $L^p(\lambda_\gamma)$-normalized extremizing sequence $\{f_n\}$ of $\scriptE_\gamma:L^p(\lambda_\gamma) \to L^q$, there exists a sequence of symmetries $\{S_n\}$ of $\scriptE_\gamma$ such that along a subsequence, either $\{S_n f_n\}$ converges in $L^p(\lambda_\gamma)$ to an extremizer, or there exists a sequence $\delta_n \searrow 0$ such that $\|S_n f_n - f_{\delta_n}\|_{L^p(\lambda_\gamma)} \to 0$, with $f_{\delta_n}$ as in \eqref{E:fdelta}, where $f,g$ are $\scriptE_{\gamma_0}$ extremizers with $|\scriptE_{\gamma_0}f| = |\scriptE_{\gamma_0}g|$, a.e on $\Rd$. The latter case can occur if and only if $B_{\gamma,p} = B_{\gamma,p}^{\rm{conc}}$.  
\end{theorem}

Let $\scriptR_\gamma$ denote the Fourier restriction operator, 
$
\scriptR_\gamma g(t) := \hat g(\gamma(t))$,  $t \in \R$.  Since $R_\gamma = \scriptE_\gamma^*$, for $r = (\frac {(d^2 + d) s}2)' < 1 + \frac 2{d^2 + d}$, 
$\|\scriptR_\gamma\|_{L^r \to L^s(\lambda_\gamma)} = B_{\gamma, s'}$.

\begin{corollary}
Let $r = (\frac {(d^2 + d) s}2)' < 1 + \frac 2{d^2 + d}$ and $\gamma : \R \to \Rd$ be a monomial. Then $\|\scriptR_\gamma\|_{L^r \to L^s(\lambda_\gamma)} \geq B_{\gamma,s'}^{\rm{conc}}$. Additionally, for $r > 1$, every extremizing sequence of $\scriptR_\gamma : L^r(\Rd) \to L^s(\lambda_\gamma)$ possesses a subsequence that converges, after application of symmetries of $\scriptR_\gamma$, in $L^r$ if and only if $\|\scriptR_\gamma\|_{L^r \to L^s(\lambda_\gamma)} > B_{\gamma,s'}^{\rm{conc}}$.
\end{corollary}

The (omitted) proof is a direct application of Theorem~\ref{T : existence} and duality,  \cite{BiswasStovall20, StovallParab}.


\section{Outline of the proof}


In Section~\ref{S:lower bound}, we prove the lower bound in Theorem~\ref{T : conc energy lower bound}.  Two results in this section, Lemma~\ref{L:op convergence} and Lemma~\ref{L : limiting constant} will also be of use in later sections and in the (forthcoming) general polynomial case.  Lemma~\ref{L:op convergence} states that extensions of a fixed function with respect to a convergent sequence of (general) polynomial curves converge in $L^p$, under certain natural hypotheses on the function, while Lemma~\ref{L : limiting constant} clarifies the case of equality in an upper bound arising in \cite{FlockStovall22}, which allows us to prove the sharp dichotomy condition of Theorem~\ref{T : existence}, without the gap appearing in the statement of~\cite[Theorem $1.5$]{FlockStovall22}. 

Section~\ref{S:chips} establishes, for $\gamma$ a general polynomial curve, a rough decomposition of an arbitrary $L^p(\lambda_\gamma)$ function $f$ into disjoint ``chips'' which are majorized by $L^p(\lambda_\gamma)$-normalized characteristic functions of intervals (Proposition~\ref{P : uniform dyadic decay}).  We identify the first chip (Proposition~\ref{P:positive bound extn}) using an argument inspired by B\'egout--Vargas's efficient implementation \cite{BegoutVargas07} of the bilinear-to-linear approach (from \cite{TaoVargasVega98}) to extension estimates for the paraboloid. As we saw in \cite{BiswasStovall20}, complications arise in adapting this strategy to curves, due to both geometric considerations in the higher codimension case and algebraic relations among the exponents.  In the general monomial (or polynomial) case, we encounter further difficulties due to deficiencies of the symmetry group of the curve. A particular challenge is in carrying out a Whitney-like decomposition of the $d$-fold sum of the curve into sets that are well-approximated by finitely overlapping convex sets. We adapt an argument of Steinig~\cite{Steinig71}, making it more quantitative (Lemma~\ref{L:convex approx}), to circumvent this difficulty (Lemma~\ref{L:convex approx}). 

In Section~\ref{S:profile}, we decompose a bounded sequence having good frequency localization to a fixed interval on a convergent sequence of curves (such as a rescaled sequence of chips) into a sum of profiles whose extensions possess good spatial localization, together with a manageable error term whose extension is negligible (Proposition~\ref{P : profiles}).  

In Section~\ref{S:chunks}, we sort the rough chips arising in Section~\ref{S:chips} into ``chunks'' comprised of chips localized to (asymptotically) commensurable intervals.  We prove (Proposition~\ref{P:freq loc}) that an extremizing sequence may possess only one chunk by leveraging convexity ($q>p$) and asymptotic orthogonality (Lemma~\ref{L:no bad k}) of incommensurable chips.  Thus an extremizing sequence is well-approximated by functions majorized by (essentially) $L^p(\lambda_\gamma)$-normalized characteristic functions of bounded intervals.  The advantage, which will be even more significant in the general polynomial case, of having previously decomposed our chips into profiles with coherent extensions in Section~\ref{S:profile} is that it permits a soft approach to the asymptotic orthogonality that avoids use of the pointwise geometric inequality \eqref{E:geom ineq}, which is only valid under hypotheses that distant chips may not meet. 

In Section~\ref{S:Lp convergence}, we prove convergence of an extremizing sequence $\{f_n\}$ in the case that the chunks rescale, via the curve's dilation group, to functions adapted to a fixed interval.   

The alternative to convergence is that the chunks rescale, via the curve's dilation group, to a sequence concentrating at 1 or $\{\pm 1\}$, depending on parity.  The phenomenon that for Fourier restriction operators concentration to antipodal points might occur was observed in~\cite{christ2012existence} for the case of restriction to unit sphere in $\R^3$. However, the fact that $q$ is even plays a major role in~\cite{christ2012existence}: it enables one to consider only non-negative antipodally symmetric functions, thereby eliminating the possibility of non-compactness due to modulation. This is not the case for $\scriptE_{\gamma}$. In~\cite{Frank_etl_2016} ($p=2$) and~\cite{FlockStovall22} ($p\neq 2$) this was addressed for the unit sphere. Our geometry is rather different from that of the sphere due to the higher codimension, the lack of compactness of the underlying manifold, and the fact that we consider all monomial curves (rather than a specific hypersurface).  

We believe that many of the techniques from this article may be of use in the study of concentration-compactness phenomena for restriction/extension operators associated to more general manifolds.  Indeed, in a forthcoming article we plan to extend our methods to produce a sharp dichotomy type result for arbitrary polynomial curves. However, the much richer geometry of such curves, including the deficient symmetry group (consisting only of modulations),  the richer near-symmetry group (including near-dilations adapted to several different monomials), and the possibility of multiple (more than two) ``antipodal'' points, leads to new obstructions, whose resolution take us beyond the scope of this article.

\subsection*{Notation} The curve $\gamma$, typically obeying \eqref{assumptions on gamma}, though occasionally a (general) polynomial of degree at most $N$, will be clear from the context; thus, $\vec l :=\vec l_\gamma$ (when $\gamma$ is a monomial) and $T(a):=T_\gamma(a)$, the torsion matrix from \eqref{E:def T} will be understood.  As we have little need to for explicit mention of the true dilations of $\gamma$, we denote $D_\lambda:=D_\lambda^{\vec l_0}$ as in \eqref{E:def T}.  We will abuse notation, conflating the the measure $\lambda_\gamma\,dt$ with its Radon--Nikodym derivative $\lambda_\gamma$; $\int_E \lambda_\gamma(t) \, dt$ is thereby shortened to  $\lambda_\gamma (E)$, for $E \subset \R$.   We use the norms $\|\cdot\|_q:=\|\cdot\|_{L^q(\R^d)}$, $\|\cdot\|_{L^p(\lambda_\gamma)}$, and (occasionally, for brevity)  $\|\cdot\|_p:=\|\cdot\|_{L^p(\R;dt)}$.    

Admissible constants, denoted by $C$ if they are sufficiently large and $c$ if they are sufficiently small, may depend on the dimension $d$, the Lebesgue exponents $p, q$, and the polynomial (monomial, in most cases) degree and may change from line to line. We use the standard notation $A \lesssim B$ to mean $A \leq C B$, where $C$ is admissible. We write $A \sim B$ to denote $A \lesssim B$ and $B \lesssim A$. Dependence of constants on inadmissible parameters will be indicated by additional subscripts.

\subsection*{Acknowledgements}
The first named author is supported by SwarnaJayanti Fellowship grant SB/SJF/2019-20/14 (PI : Apoorva Khare) and by a C V Raman Postdoctoral Fellowship. The second named author is supported in part by NSF grant DMS-1653264 and the Wisconsin Alumni Research Foundation (WARF). Part of the work was carried out when the authors visited Hausdorff Research Institute for Mathematics to participate in the Trimester program ``Interactions between Geometric measure theory, Singular integrals, and PDE" funded by Deutsche Forschungsgemeinschaft (DFG, German Research Foundation) under Germany's Excellence Strategy-EXC-2047/1-390685813 in spring $2022$. We thank the organizers for excellent hospitality. We would also like to thank Ramesh Manna for a number of interesting discussions related to this work. We thank the anonymous referee for valuable comments.


\section{Proof of Theorem~\ref{T : conc energy lower bound}: Lower bound for the operator norm} \label{S:lower bound}


In this section, we will prove Theorem~\ref{T : conc energy lower bound} by rescaling an extremizer for extension from the moment curve to concentrate either at a single point or symmetrically at antipodal points (depending on parity).  We begin with a convergence result for general polynomial curves, which will be of use in what follows.  

\begin{lemma}\label{L:op convergence}
Let $\gamma_n$ be a sequence of polynomials, each of degree at most $N$, and assume that $\gamma_n \to \gamma$ in the sense that the coefficients of the $\gamma_n$ converge to those of $\gamma$. Then for each $q > \frac{d^2 + d + 2}2$ and $f$ bounded and measurable with compact support, $\scriptE_{\gamma_n}f \to \scriptE_\gamma f$ in $L^q$.  
\end{lemma}

We require that $f$ be bounded with compact support (rather than taking an arbitrary $L^p$ function) because the measure $\lambda_{\gamma_n}$ (and hence $L^p(\lambda_{\gamma_n})$) varies with $n$.  

One difference to note between the case of curves and the case of elliptic hypersurfaces is that, although the Fourier extension of a smooth function on an elliptic hypersurface decays isotropically at a rate that places it in every $L^q$ space into which the extension operator is conjectured to be bounded, smooth functions on curves decay anisotropically. Thus, the decay rate is not stable under perturbations of the curve, so there is no dominating function lying in many of the $L^q$ spaces of interest.

\begin{proof}
Since $\{B_{\gamma_n, p}\}$ is bounded (by Theorem~\ref{T:extn polynomial}), by density arguments it suffices to prove the claimed convergence for smooth $f$ with compact support. 

By the equivalence of all norms on the finite dimensional vector space of polynomials of degree at most $N$, $\gamma_n \to \gamma$ in $C^N([-R,R])$ for every $R>0$.  Therefore, by H\"older's inequality, $\scriptE_{\gamma_n}f \to \scriptE_\gamma f$ locally uniformly.  We will upgrade this local uniform convergence to uniform convergence. 

If $\lambda_\gamma \equiv 0$, uniform convergence is immediate because $\lambda_{\gamma_n} \to 0$ uniformly on the support of $f$.  If $\lambda_\gamma \not\equiv 0$, then $\lambda_\gamma$ vanishes on a finite set, and by density arguments, we may assume that $\lambda_\gamma$ does not vanish on the support of $f$. Thus for $n$ sufficiently large, $\lambda_{\gamma_n} \sim 1$ on the support of $f$. Hence by stationary phase, $|\scriptE_{\gamma_n} f| \lesssim_{f, \gamma} \jp{x}^{- \frac 1d}$, whence the convergence of  $\scriptE_{\gamma_n} f$ to $\scriptE_\gamma f$ is indeed uniform.   

Finally, we pass from uniform to $L^q$ convergence.  Choose some $\frac{d^2+d+2}d < q_0 < q$ and set $\theta := \tfrac{q_0}q$ and $p_0 := (\tfrac{2 q_0}{d^2 + d})'$. Since $\|\scriptE_{\gamma_n} f - \scriptE_\gamma f\|_{q_0}$ is uniformly bounded, by H\"older's inequality and the above uniform convergence,
\begin{align*}
    \|\scriptE_{\gamma_n} f - \scriptE_\gamma f\|_{q} \lesssim_{f, p_0} \|\scriptE_{\gamma_n} f - \scriptE_\gamma f\|_{\infty}^{1 - \theta} \to 0.  
\end{align*}
\end{proof}

The following notation will be useful in constructing certain approximating polynomials.  
Let $\gamma : \R \to \Rd$ be a monomial of degree $N$ and $a \neq 0$. For each positive $\delta$ the polynomial $\gamma_{(a, \delta)} : \R \to \Rd$ is defined by
\begin{align}\label{D : gamma(a, delta)}
\gamma_{(a, \delta)}(t) := D_{\delta^{- 1}} T(a)^{- 1} (\gamma(a + \delta t) - \gamma(a)),
\end{align}
where $T(a)$ is the torsion matrix from \eqref{E:def T}.  Taylor expanding,  
\begin{equation}\label{E : local approx}
\|\gamma_{(a, \delta)} - \gamma_0\|_{C^N[- R, R]} \xrightarrow[\delta \searrow 0]{}0, \qquad a \neq 0,
\end{equation}
while a few simple natural changes of variables and basic linear algebra yield
\begin{equation} \label{E : scriptE(a, delta) preserves L^q}
\begin{gathered}
\scriptE_\gamma f(x) = |\det(T(a)D_\delta)|^{\frac1q} e^{ix\cdot\gamma(a)}\scriptE_{\gamma_{(a, \delta)}}\bigl((\delta\lambda_\gamma(a))^{\frac1p}f(a+\delta \cdot)\bigr)\bigl(D_\delta T(a)^t x\bigr)
\\
\|f\|_{L^p(\lambda_\gamma)} = \bigl\|\bigl(\delta \lambda_\gamma(a)\bigr)^{\frac1p}f(a+\delta\cdot)\bigr\|_{L^p(\lambda_{\gamma_{(a, \delta)}})},
\end{gathered}
\end{equation}
whenever $q = \tfrac{d^2 + d}2 p' > p$.

With these preliminaries in place, we are ready to prove the theorem.

\begin{proof}[Proof of Theorem~\ref{T : conc energy lower bound}]
We begin with the case $\gamma$ neither odd nor even.
Let $f$ with $\|f\|_p = 1$ be an extremizer of $\scriptE_{\gamma_0} : L^p \to L^q$. For each $0 < \delta < R^{-1}$, we define $f_\delta^R \in L^p(\lambda_\gamma)$ by
\begin{equation}\label{weak sequence for even l}
f_\delta^R(t):=\delta^{-1/p}\lambda_\gamma(t)^{-1/p}f\bigl(\tfrac{t-1}\delta\bigr)\rchi_{[-R,R]}\bigl(\tfrac{t-1}\delta\bigr).
\end{equation} 

Since $q = \frac{d^2 + d}{2} p'$ and $\|f_\delta^R\|_{L^p(\lambda_\gamma)} = \|f \rchi_{[- R, R]}\|_p$, by identity~\eqref{E : scriptE(a, delta) preserves L^q} at $a = 1$ and a change of variables
\begin{align}
\|f \rchi_{[- R, R]}\|_p B_{\gamma, p} & \notag \geq \limsup_{\delta \searrow 0} \delta^{- \frac 1{p'}} \|\scriptE_\gamma f_\delta^R(D_{\delta^{- 1}} \cdot)\|_q
\\\notag
& = \lambda_\gamma (1)^{\frac 1p} \limsup_{\delta \searrow 0} \|\scriptE_{\gamma_{(1, \delta)}} \Big(f \rchi_{[- R, R]} \lambda_\gamma(1 + \delta \cdot)^{- \frac 1p}\|_q = \|\scriptE_{\gamma_0} (f \rchi_{[- R, R]})\|_q. 
\end{align}
The final equality follows from Lemma~\ref{L:op convergence}, the uniform convergence~\eqref{E : local approx} at $a = 1$, and the triangle inequality. Letting $R \to \infty$, $f \rchi_{[- R, R]} \to f$, an extremizer of $\scriptE_{\gamma_0}$, so
$$
B_{\gamma, p} \geq B_{\gamma_0, p}.
$$ 

Though $f_\delta^R$ is not quite a truncation of the $f_\delta$ defined in Theorem~\ref{T : conc energy lower bound}; by continuity of $\lambda_\gamma$ and the dominated convergence theorem, we see that
$$
\lim_{R \to \infty} \lim_{\delta \searrow 0} \|\lambda_\gamma(1)^{-1/p}f_\delta - f_\delta^R\|_{L^p(\lambda_\gamma)} = 0.
$$
Since $f_\delta^R$ is extremizing as $\delta \searrow 0$ and $R \to \infty$, $f_\delta$ is extremizing as $\delta \searrow 0$, 
so Theorem~\ref{T : conc energy lower bound} follows in the ``neither'' case.

Now we turn to the cases when $\gamma$ is either odd or even.

\fbox{Even $\gamma$ :} Analogously with the previous case, it suffices to work with a slight modification and truncation of $f_\delta$,
$$
f_\delta^R(t):=\delta^{-1/p}\bigl[f\bigl(\tfrac{t-1}\delta\bigr)\rchi_{[-R,R]}\bigl(\tfrac{t-1}\delta\bigr) + f\bigl(\tfrac{-(t+1)}\delta\bigr)\rchi_{[-R,R]}\bigl(\tfrac{t+1}\delta\bigr)\bigr]\lambda_\gamma(t)^{-1/p}, \quad R<\delta^{-1}.
$$
By~\eqref{E : scriptE(a, delta) preserves L^q}, for $R < \delta^{-1}$,
\begin{align*}
    2^{1/p} B_{\gamma,p}\|f\rchi_{[-R,R]}\|_{L^p(dt)} & = B_{\gamma,p}\|f_\delta^R\|_{L^p(\lambda_\gamma)} \geq \|\scriptE_\gamma f_\delta^R\|_q\\
    & = 2 \lambda_\gamma (1)^{\frac 1p} \|\scriptE_{\gamma_{(1, \delta)}}\bigl(f \rchi_{[-R,R]} \lambda_\gamma (1 + \delta \cdot)^{- \frac 1p}\bigr)\|_{L^q_x}.
\end{align*}
By analogous computations to the previous case, the limit as $\delta \searrow 0$ of the right  side of the preceding is $2\|\scriptE_{\gamma_0}(f\rchi_{[-R,R]})\|_q$, which tends to $2B_{\gamma_0,p}$, as $R \to \infty$.  Thus, $B_{\gamma,p} \geq 2^{1/p'}B_{\gamma_0, p}$ and $f_\delta^R$ is extremizing as $\delta \searrow 0$ and $R \to \infty$, as claimed.  

\fbox{Odd $\gamma$ :} Odd curves possess ``antipodal'' points, distinct points in the image whose first $d$ derivatives match after a sign change.  This leads to spatial overlap of extensions oscillating at different frequencies, whose interactions lead us to Lemma~\ref{L : limiting constant} below.  We recall the function $\psi_{p,q}$ and its maximum $\Psi_{p,q}$ from \eqref{E : define Psi_p,q}. 

\begin{lemma}\label{L : limiting constant}
Let $1 < p \leq q$ and $0 \neq v \in \Rd$. Then for each $F, G \in L^q(\Rd)$, 
\begin{equation}\label{ineq for lim int}
\lim_{\lambda \to \infty} \|F(x) + e^{i D_{\lambda} x \cdot v} G(x)\|_{L_x^q}^q
\leq  \Psi_{p, q} (\|F\|_q^p + \|G\|_q^p)^{\frac qp}.
\end{equation}
The limit on the left side of~\eqref{ineq for lim int} always exists, and equality holds if and only if $|F|=\alpha|G|$, a.e., or $|G| = \alpha |F|$, a.e., for some maximum point $\alpha \in [0,1]$ of $\psi_{p,q}$.  
\end{lemma}

\begin{proof}Applying~\cite[Lemma $5.2$]{Allaire92} (see also~\cite[section $6$] {Frank_etl_2016}), we have
\begin{equation}\label{E : lim constant value}
  \lim_{\lambda \to \infty} \|F(x) + e^{i D_{\lambda} x \cdot v} G(x)\|_{L^q_x}^q \notag = \frac1{2 \pi} \int_{\Rd} \int_{- \pi}^{\pi} |F(x) + e^{i \theta} G(x)|^q \, d\theta \, dx.  
\end{equation}
Changing variables in $\theta$, we may assume that $F=|F|, G=|G|$. Since $\psi_{p, q}(0) = 1$, decomposing $\Rd$ based on the relative sizes of $F$ and $G$ implies that the above is
\begin{align*}
& \int_{\{F G = 0\}} (F^p + G^p)^{\frac qp} \psi_{p,q}(0)\, dx + \int_{\{FG \neq 0\}} (F^p + G^p)^{\frac qp} \psi_{p,q}\bigl(\min\{\tfrac FG,\tfrac GF\}\bigr)\,dx 
\\
&\qquad \leq  \Psi_{p, q}\|F^p + G^p\|_ {\frac qp}^{\frac qp} 
\leq \Psi_{p, q}(\|F^p\|_{\frac qp} + \|G^p\|_ {\frac qp})^{\frac qp}   = \Psi_{p, q} (\|F\|_q^p + \|G\|_q^p)^{\frac qp}  .
\end{align*}
The first inequality follows from H\"older, and the second inequality follows from the triangle inequality in $L^{\frac qp}$. By the sharp forms of the H\"older and Minkowski inequalities, the two inequalities are both equalities if and only if either $|F| = \alpha|G|$ a.e. or $|G| = \alpha|F|$ a.e., where $\alpha \in [0,1]$ is a maximum point of $\psi_{p,q}$.  
\end{proof}

Now we are ready to complete the proof of Theorem~\ref{T : conc energy lower bound} in the case that $\gamma$ is odd.  Analogously with the case of even $\gamma$ we consider
$$
f_\delta^R(t):=\delta^{-\frac1p}\bigl[f\bigl(\tfrac{t-1}\delta\bigr)\rchi_{[-R,R]}\bigl(\tfrac{t-1}\delta\bigr) + \alpha \overline{g\bigl(\tfrac{-(t+1)}\delta\bigr)}\rchi_{[-R,R]}\bigl(\tfrac{t+1}\delta\bigr)\bigr]\lambda_\gamma(t)^{-\frac1p}, \:\: R<\delta^{-1},
$$
and observe that
$$
B_{\gamma,p} (\|f\rchi_{[-R,R]}\|_{L^p(dt)}^p + \alpha^p \|g\rchi_{[-R,R]}\|_{L^p(dt)}^p)^{\frac 1p} = B_{\gamma,p} \|f_\delta^R\|_{L^p(\lambda_\gamma)} \geq \|\scriptE_\gamma f_\delta^R\|_q.
$$
As $R \to \infty$, the above LHS goes to $(1 + \alpha^p)^{\frac 1p} B_{\gamma, p}$.

By \eqref{E : local approx} and $q = \tfrac{d(d+1)}2 p'$,
\begin{align*}
    \scriptE_\gamma f_\delta^R(x) = & \lambda_\gamma (1) \delta^{\frac{d(d+1)}{2q}}\bigl[e^{ix\cdot \gamma(1)} \scriptE_{\gamma_{(1, \delta)}}\bigl(f \rchi_{[-R,R]}\lambda_\gamma (1 + \delta \cdot)^{- \frac 1p} \bigr)(D_\delta T(1)^t x) \\
    &+\alpha e^{ix\gamma(-1)} \scriptE_{\gamma_{(- 1, \delta)}}\bigl(\overline{g(-\cdot)}\rchi_{[-R,R]}\lambda_\gamma (- 1 + \delta \cdot)^{- \frac 1p} \bigr)(D_\delta T(-1)^t x)\bigr].
\end{align*}
Since $\gamma$ is odd, $T(-1)^{-1}T(1) = -D_{- 1}$. By the change of variables formula,
\begin{align*}
    &\|\scriptE_\gamma f_\delta^R\|_q = \lambda_\gamma (1)^{\frac 1p}  \|\scriptE_{\gamma_{(1, \delta)}}\bigl(f\rchi_{[-R,R]}\lambda_\gamma (1 + \delta \cdot)^{- \frac 1p} + \alpha e^{-2iD_{\delta^{-1}}^{\vec l} x \cdot \gamma(1)} \\
    & \hspace{3 cm} \scriptE_{\gamma_{(- 1, \delta)}}\bigl(\overline{g(-\cdot)}\rchi_{[-R,R]}\lambda_\gamma (- 1 + \delta \cdot)^{- \frac 1p}\bigr)(-D_{- 1} x)\|_{L^q_x}.
\end{align*}
The limit as $\delta \searrow 0$ of the RHS is 
\begin{align*}
    &\lim_{\delta \searrow 0} \| \scriptE_{\gamma_0}(f\rchi_{[-R,R]})(x) + e^{-2iD_{\delta^{-1}}^{\vec l}x \cdot \gamma(1)}\alpha \scriptE_{\gamma_0}(\overline{g(-\cdot)}\rchi_{[-R,R]})(-D_{- 1}x)\|_{L^q_x}\\
    &\qquad = \lim_{\delta \searrow 0} \|\scriptE_{\gamma_0}(f\rchi_{[-R,R]})(x) + e^{-2iD_{\delta^{-1}}^{\vec l} x \cdot \gamma(1)} \alpha \overline{\scriptE_{\gamma_0}(g\rchi_{[-R,R]})(x)}\|_{L^q_x}\\
    &\qquad = \lim_{\delta \searrow 0} \|\scriptE_{\gamma_0} f(x) + e^{-2iD_{\delta^{-1}}^{\vec l} x \cdot \gamma(1)} \alpha \overline{\scriptE_{\gamma_0} g(x)}\|_{L^q_x} + o_R(1)\\
    &\qquad = (1+\alpha^p)^{1/p} \Psi_{p,q}^{1/q} \, B_{\gamma_0, p} + o_R(1).
\end{align*}
Letting $R \to \infty$ yields $B_{\gamma,p} \geq \Psi_{p,q}^{1/q}B_{\gamma_0,p}$ and $f_\delta^R$ (and hence $f_\delta$) is extremizing in the case of equality.  
\end{proof}

Having completed the proof of Theorem~\ref{T : conc energy lower bound}, we turn to that of Theorem~\ref{T : existence}.


\section{Improved $L^p$ estimate} \label{S:chips}


Our goal in this section is to improve on the $L^p(\lambda_\gamma) \to L^q$ inequality for $\scriptE_\gamma$, by replacing the $L^p(\lambda_\gamma)$ norm with something slightly smaller (Proposition~\ref{P:positive bound extn}). The primary purpose of this improved estimate is to replace each near extremizer for $\scriptE_\gamma$ by a few pieces that capture most of the $L^p(\lambda_\gamma)$-mass (Proposition~\ref{P : uniform dyadic decay}). Since all results in this section extend to general polynomials without any significant change in the argument, for the remainder of this section $\gamma$ will denote an arbitrary polynomial mapping from $\R$ into $\R^d$, of degree at most $N$.

For $m \in \Z$, let $\scriptI_m$ denote the set of intervals $I \subseteq \R$ with $\lambda_\gamma \sim c_I$ on $I$ and $c_I |I| \sim 2^m$ (thus, $\lambda_\gamma(I) \sim 2^m$), and let $\scriptI :=\bigcup_{m \in \Z} \scriptI_m$.   

\begin{proposition} \label{P:positive bound extn}
Let $\gamma:\R \to \R^d$ be a polynomial of degree at most $N$ and $q := \tfrac{d^2 + d}2 p' > \tfrac{d^2 + d + 2}2$. There exist $c=c_p>0$ and $0 < \theta = \theta_p < 1$ such that 
\begin{equation} \label{E:positive bound extn}
\|\scriptE_\gamma f\|_q \lesssim \sup
2^{-c n} \Big(\|f \rchi_{I \cap \{|f|\leq 2^n 2^{- m/p} \|f\|_{L^p(\lambda_\gamma)}\}}\|_{L^p(\lambda_\gamma)}\Big)^\theta \|f\|_{L^p(\lambda_\gamma)}^{1 - \theta}, 
\end{equation}
for every $f \in L^p(\lambda_\gamma)$.  The supremum in \eqref{E:positive bound extn} is taken over $m \in \Z$, $I \in \scriptI_m$, $n \geq 0$.
\end{proposition}

\begin{proof}[Proof of Proposition~\ref{P:positive bound extn}]
It is shown in \cite{Dendrinos_Wright_2010} 
that there exists a decomposition of $\R$ as a union of intervals, 
\begin{equation}\label{decompose R}
\R  = \bigcup_{l = 1}^{C_{d, N}} I_j
\end{equation}
such that the geometric inequality
\begin{equation} \label{E:geom ineq}
|\det(\gamma'(t_1), \ldots, \gamma'(t_d))| \gtrsim \prod_{j = 1}^d |L_\gamma(t_j)|^{1/d} \prod_{1 \leq i < j \leq d}|t_i - t_j|
\end{equation} 
holds for $(t_1,\ldots,t_d) \in I_j^d$, for each $j$. Since $(L_\gamma)^2$ is a polynomial, it has a bounded number of critical points, and so we may assume that $|L_\gamma|$ is monotone on each $I_j$; otherwise, we further subdivide $I_j$ into a bounded number of intervals.  

By~\eqref{decompose R} and the triangle inequality, validity of \eqref{E:positive bound extn} for each $f \rchi_{I_j}$ will imply its validity for $f$, so for the remainder of the argument, we proceed with $j$ fixed.

Next, we reduce matters to consideration of intervals on which $L_\gamma$ is comparable to a constant. Let $A_k := \{2^{k - 1} < |L_\gamma| \leq 2^k\}$ and $A_{j, k} := A_k \cap I_j$. By construction, each $A_{j, k}$ is an interval. 

\begin{lemma} \label{L:reduce dyadic}
There exists $0 < \theta = \theta_p < 1$ such that
$$
\|\scriptE_\gamma (f \rchi_{I_j})\|_q \lesssim \max_{k \in \Z} \|\scriptE_\gamma (f \rchi_{A_{j, k}})\|_q^{\theta} \|f\|_{L^p(\lambda_\gamma)}^{1 - \theta}.
$$
\end{lemma}
\begin{proof}[Proof of Lemma~\ref{L:reduce dyadic}]
We define  $M := \lceil q\rceil + 1$. Since $(a + b)^\alpha \leq a^\alpha + b^\alpha$ for $a,b \geq 0$ and $ 0 < \alpha \leq 1$, a bit of arithmetic gives
\begin{align}\label{get ordered pair}
\|\scriptE_\gamma (f \rchi_{I_j})\|_q^q & \notag = \int |\sum_{\vec{k} \in \Z^M} \prod_{l = 1}^M \scriptE_\gamma (f \rchi_{A_{j, k_l}})|^{\frac qM} \leq \sum_{\vec{k} \in \Z^M} \int |\prod_{l = 1}^M \scriptE_\gamma (f \rchi_{A_{j, k_l}})|^{\frac qM} dx
\\
& \lesssim \sum_{m = 0}^\infty \,\, \sum_{k_1 \leq k_2 \leq \cdots \leq k_M = k_1 + m} \|\prod_{l = 1}^M \scriptE_\gamma (f \rchi_{A_{j, k_l}})\|_{\frac qM}^{\frac qM}.
\end{align}
By H\"older's inequality and~\cite[Lemma~4.1]{Stovall16}, for $k_1 \leq \cdots \leq k_M = k_1 + m$, 
\begin{align*}
\|\prod_{l = 1}^M \scriptE_\gamma (f \rchi_{A_{j, k_l}})\|_{\frac qM} & \lesssim 2^{-c m} \prod_{l = 1}^M \|f \rchi_{A_{j, k_l}}\|_{L^p(\lambda_\gamma)},
\end{align*}
for some admissible $c>0$. Therefore, by the H\"older and arithmetic-geometric mean inequalities, and the fact that an interval of length $m$ contains at most $(m+1)^M$ ordered integer sequences of length $M$, RHS~\eqref{get ordered pair} is bounded by a constant times
\begin{align*}
&(\sup_k \|\scriptE_\gamma(f \rchi_{A_{j, k}})\|_q^{q - p}) \sum_{m = 0}^\infty 2^{-c m} \sum_{k_1 \leq k_2 \leq \cdots \leq k_M = k_1 + m} \,\, \sum_{l = 1}^M \|f \rchi_{A_{j, k_l}}\|_{L^p(\lambda_\gamma)}^p
\\
& \qquad \lesssim (\sup_k \|\scriptE_\gamma (f \rchi_{A_{j, k}})\|_q^{q - p}) \sum_{m = 0}^\infty 2^{-c m} (m + 1)^M \sum_k \|f \rchi_{A_{j, k}}\|_{L^p(\lambda_\gamma)}^p
\\
& \qquad 
\lesssim (\sup_k \|\scriptE_\gamma (f \rchi_{A_{j, k}})\|_q^{q - p}) \|f \rchi_{I_j}\|_{L^p(\lambda_\gamma)}^p.
\end{align*}
\end{proof}

It remains to prove \eqref{E:positive bound extn} for functions $f$ supported on a single $A_{j,k}$. We will reduce this problem to proving \eqref{E:positive bound extn} for a small perturbation of the moment curve $\gamma_0$.  Our initial reductions follow those from~\cite[Section 2]{Stovall16}, so we will be brief.  

By \cite{BiswasStovall20}, it suffices to consider the case when $L_\gamma$ is not constant, so $A_{j, k}$ is a bounded interval.  By rescaling and translating $A_{j, k}$, we may assume that $A_{j, k} = [0, 1]$. By affine invariance, we may further assume that $L_\gamma \sim 1$ on $[0, 1]$. By~\cite[Lemma $2.2$]{Stovall16}, after an additional affine map, we may assume that $\|\gamma\|_{C^N([0,1])} \lesssim 1$; in fact, after a finite decomposition and a further affine transformation, rescaling, and translation, we may assume that 
\begin{equation} \label{E:moment plus err}
\gamma(t) = \gamma_0 + \tilde\gamma, \qtq{with} \tilde \gamma^{(j)}(0) = 0, \:j = 0, \ldots, d, \,\, \text{and} \,\, \|\tilde\gamma\|_{C^N([0, 1])} < \eps_{d, N},
\end{equation}
with $\eps_{d, N}$ sufficiently small for later purposes.  

We consider the cube $Q := [0, 1]^d \subset \Rd$, which contains the diagonal 
$$
\Delta := \{(t, \ldots, t) \in \R^d : t \in [0, 1]\}.  
$$
Let $K$ be a fixed positive integer, sufficiently large for later purposes. We will choose $K$ to depend on $d, N$, but other constants within the proof will not be allowed to depend on $K$. Let $\scriptR_m$ denote the set of all dyadic cubes $\tau$ in $[0, 1]^d$ with side-length $2^{- m}$ and
$$
2^{- m + K} \leq \dist(\tau, \Delta) \leq 2^{- m + 2K}.
$$
Set $\scriptR:=\bigcup_{m \geq K} \scriptR_m$.  Then 
$$
Q \setminus \Delta = \bigcup_{R \in \scriptR} R, 
$$
and each cube in the union on the right intersects only a bounded number of other cubes, which all have comparable side lengths and distance to $\Delta$.  Therefore,
$$
\|\scriptE_\gamma (f \rchi_{[0, 1]})\|_q^d = \|\sum_{R \in \scriptR} \prod_{j = 1}^d \scriptE_\gamma f_{R, j}\|_{\frac qd},
$$
where $f_{R, j}$ is supported on the projection of the cube $R$ onto the $j$-th coordinate and $|f_{R, j}| \lesssim |f|$. We define $f_R := (\prod_{j = 1}^d \scriptE_\gamma f_{R, j})\widehat{\:}$. Thus $f_R$ is supported in $\Gamma(R)$, where 
\begin{equation}\label{define Gamma}
\Gamma(t_1, \ldots, t_d) := \gamma(t_1) + \cdots + \gamma(t_d).
\end{equation}

From a lemma of Steinig~\cite{Steinig71}, under hypothesis \eqref{E:moment plus err}, for $\eps_{d, N}$ small enough, $\Gamma$ is $d!$-to-one and, in particular, one-to-one on the subsets $\{0 \leq t_{\sigma(1)} \leq \cdots \leq t_{\sigma(d)} \leq 1\}$, $\sigma \in S_d$ (the symmetric group on $d$ letters). Therefore, $\Gamma(R)$ and $\Gamma(R')$, $R, R' \in \scriptR$ only intersect when $R$ and $R'$ intersect after reordering the coordinates of one of them. Consequently, the $f_R$, $R \in \scriptR$ have finitely overlapping frequency supports.   

Our next step is to apply an analogue of an almost-orthogonality lemma of Tao-Vargas-Vega~\cite[Lemma $6.1$]{TaoVargasVega98}.  For this, we need something slightly stronger than bounded overlap of the $\Gamma(R)$'s, namely bounded overlap of some convex approximations of the $\Gamma(R)$'s. Let 
\begin{align*}
& \scriptA:=\bigl\{(m, \vec n, k) \in \R^{1 + d + 1} : K \leq m \in \Z, 0 \leq k \in 2^{- m} \Z, \vec n \in 2^{-m}([0, 2^{2K}] \cap \Z)^d,
\\
& \hspace{6 cm} 0 \in \{n_1, \ldots, n_d\}, \qtq{and} |\vec n| \geq 2^{- m + K}\bigr\}.
\end{align*}
We observe, in particular, that for each $m,k$, there are at most $2^{2Kd}$ $\vec n$ for which $(m,\vec n,k) \in \scriptA$.

Each element of $\scriptR$ may be written in the form
$$
R(m, \vec n, k) := 2^{- m} [0, 1]^d + \vec n + \vec k, \qquad (m, \vec n, k) \in \scriptA, \qquad \vec k:=(k,\ldots,k).  
$$
We define convex sets
\begin{gather*}
Q(m, \vec n, k) := \Gamma(\vec n + \vec k) + \sum_{j = 1}^d T (n_j + k) P_{2^{- m}} \subset \Rd, \qtq{where} P_\lambda := \prod_{j = 1}^d [- \tfrac{\lambda^j}{j!}, \tfrac{\lambda^j}{j!}].
\end{gather*}
We will use these to approximate the $R(m,\vec n,k)$.    

\begin{lemma}\label{L:convex approx}
Assume that $\gamma$ obeys~\eqref{E:moment plus err}, and let $\delta : = \eps_{d, N} 2^{- (K + 1) d}$. Then for $K = K_{d, N}$ and $L = L_{d, N}$ sufficiently large, and $m \geq K$,
\begin{equation} \label{E:convex approx}
\Gamma(R (m + 1, \vec n, k)) + T(k)P_{\delta2^{-m}} \subseteq Q(m, \vec n, k), \quad (m, \vec n, k) \in \scriptA, 
\end{equation}
and there exists a decomposition $\scriptA = \bigcup_{j = 1}^L \scriptA_j$ so that for each $j$,
\begin{equation} \label{E:finite overlap}
Q(m, \vec n, k) \cap Q(m', \vec n', k') = \emptyset, \qtq{for} (m, \vec n, k) \neq (m', \vec n', k') \in \scriptA_j.
\end{equation}  
\end{lemma}
\begin{proof}[Proof of Lemma~\ref{L:convex approx}]
Observe that $\Gamma(R(m,\vec n,k))$ and $Q(m,\vec n,k)$ remain unchanged after reordering the coordinates of $\vec n$.  Thus we may restrict ourselves to consideration of those $(m, \vec n, k)$ with $0 = n_1 \leq \ldots \leq n_d$.

We begin with the proof of~\eqref{E:finite overlap}. With $L$ to be determined, we construct subsets $\scriptA_j$, $j= 1, \ldots, L 2^L$ such that whenever $(m, \vec n, k), (m', \vec n', k') \in \scriptA_j$, either $m = m'$ or $|m - m'| \geq L$ and if $m = m'$, then either $k = k'$ or $|k - k'|\geq  2^{L - m}$. Let $(m, \vec n, k) \neq (m', \vec n', k')  \in \scriptA_j$, and suppose that $x_0 \in Q(m, \vec n, k) \cap Q(m', \vec n', k')$.  

If $m = m'$, 
\begin{align*}
x_0 &= \Gamma(\vec n + \vec k) + O(2^{- m}) = \Gamma(\vec n' + \vec k') + O(2^{- m})
\\
& = d \gamma(k) + O(2^{K - m}) = d \gamma(k') + O(2^{K - m}),
\end{align*}
where the $O(\cdot)$ terms have norm bounded by an admissible constant times their argument.  
However, since $\gamma_1'(t) \sim 1$ on $[0,1]$ by \eqref{E:moment plus err},  this implies $|k-k'| \lesssim 2^{K-m}$, a contradiction if $L \geq K+C_{N,d}$.  

Now suppose that $m < m'$.  Thus $m + L \leq m'$. Assuming $L \geq K + C_{d, N}$, and arguing as above, 
\begin{equation} \label{E:x=Qmlk=Qm'l'k'}
x_0 = \Gamma(\vec n + \vec k) + O(2^{- m}) = d \gamma(k')+ O(2^{- m}),
\end{equation}
so $|k - k'| \sim |\vec n| \sim 2^{K - m}$. We order the (distinct) elements of $\{n_1 + k, \ldots, n_d + k, k'\}$ as $t_0 < \cdots < t_l$, with $1 \leq l \leq d$. Since $n_1 = 0$ and $|\vec n| \sim 2^{K - m}$, $t_l - t_0 \sim 2^{K - m}$. Let 
\begin{equation} \label{E:def bar gamma}
\bar \gamma := D_{2^{m - K}} T(t_0)^{- 1} \gamma(2^{K - m} \cdot + t_0).
\end{equation}
Then $\bar \gamma$ also obeys \eqref{E:moment plus err}, and $D_{2^{m - K}} T(t_0)^{- 1} x_0 \in \bar Q(K, \vec n, \bar k) \cap \bar Q(m' - m + K, \vec n', \bar k')$, where $\bar k, \bar k' \leq 1$ and the $\bar Q$ are the convex sets associated to $\bar \gamma$.  

Thus we may assume that $m=K$. We may further assume, by absorbing any errors into the $O(2^{- K})$ term in \eqref{E:x=Qmlk=Qm'l'k'}, that $k' \in 2^{- K} \Z_{\geq 0}$. Thus the $t_j$ lie in $2^{- K}\Z_{\geq 0}$, and $t_0 = 0, t_l \sim 1$. By pigeonholing, for $K = K_d$ sufficiently large, there exist positive $A$ and $B$ with $A > d^2 B + C$ (with $C = C_d$ a sufficiently large constant) such that for each $j, j'$, either $|t_j - t_{j'}| < 2^{- A}$, or $|t_j - t_{j'}| > 2^{- B}$. We say that $j \sim j'$ if $|t_j - t_{j'}| < 2^{- A}$. Since $A - d > B$, this is an equivalence relation on $\{0, \ldots, l\}$. Furthermore, $t_0 = 0$ and $t_l \sim 1$ lie in two different equivalence classes. We choose one representative from each equivalence class, denoted by $0 = s_0 < \ldots < s_{l'} = 1$. Henceforth, we assume that $l' = d$; if not, we truncate $\gamma$ to its first $l'$ coordinates.  

We now adapt Steinig's argument. Approximating each $t_j$ by the $s_i$ from its equivalence class, we may rewrite \eqref{E:x=Qmlk=Qm'l'k'} as 
$$
\sum_{j=0}^d a_j \gamma(s_j) = O(2^{-A}),
$$
where the $a_j$ are integers with $\sum_j a_j = 0$.  Moreover, exactly one $a_j$ is negative (this is the coefficient of $\gamma(s_i)$, where $k' \in [s_i]$), and the rest are positive. For $1 \leq i \leq d$, with $b_i := \sum_{j = 0}^{i - 1} a_j$, the above is
\begin{equation} \label{E:lin dep}
\sum_{i = 1}^d b_i \int_{s_{i - 1}}^{s_i} \gamma'(u) \, du = O(2^{- A}).
\end{equation}
Since $b_1 \neq 0$ this implies that
$$
\det(\int_{s_{j - 1}}^{s_j} \gamma'(u_j)\, du_j)_{j = 1}^d = O(2^{- A}).  
$$
By multi-linearity of determinants, 
\begin{equation} \label{E:int of det}
\int_{s_0}^{s_1} \cdots \int_{s_{d - 1}}^{s_d} \det(\gamma'(u_1), \ldots, \gamma'(u_d)) \, du_d \, \ldots \, du_1 = O(2^{- A}).
\end{equation}
By inequality \eqref{E:geom ineq}, 
$$
|\det(\gamma'(u_1), \ldots, \gamma'(u_d))| \gtrsim \prod_{1 \leq i < j \leq d} |u_i - u_j|.
$$
In particular, on the domain in~\eqref{E:int of det} the integrand is single-signed.  Therefore,
\begin{equation} \label{E:int of vdm}
\int_{s_0}^{s_1} \cdots \int_{s_{d - 1}}^{s_d} \prod_{1 \leq i < j \leq d} |u_i - u_j| \, du_d \, \ldots \, du_1 = O(2^{- A}).
\end{equation}
On the other hand, since $s_i \geq s_{i - 1} + 2^{- B}$ for all $i$, LHS\eqref{E:int of vdm} $\gtrsim 2^{- B d^2}$, a contradiction for $C$ sufficiently large.  Thus \eqref{E:finite overlap} is proved.

Now we turn to the proof of \eqref{E:convex approx}.  Making the transformation in \eqref{E:def bar gamma}, it suffices to consider the case $m=K$.

We observe that $Q(K + 1, \vec n, k) \subset \frac 12 Q(K, \vec n, k)$, where we use the usual convention that for $E$ a symmetric convex body and $x \in \R^d$, $\frac 12 (E + x) := \frac 12 E + x$.

By making a Taylor approximation, $\Gamma(R(K+1, \vec n, k)) \subseteq Q(K + 1, \vec n, k) + B(0, d \delta)$, so it suffices to prove that $B(0, (d + 1) \delta) \subseteq Q(K+1, \vec n,k) - \Gamma(\vec n + \vec k)$.  However, by inequality~\eqref{E:moment plus err}, $\|T (t)^{- 1}\| \lesssim 1$, uniformly on $[0,1]$, and so each $T (n_j + k) P_{2^{- (K + 1)}}$ contains $B(0, (d + 1) \delta)$, when $K$ is sufficiently large.  

Finally, the proof of Lemma~\ref{L:convex approx} is complete.
\end{proof}

For each $\sigma \in S_d$, we let $\scriptR_\sigma \subset \scriptR$ denote the set of those $R \in \scriptR$ satisfying $t_{\sigma(1)} \leq \ldots \leq t_{\sigma(d)}$ for $t \in R$.
\begin{lemma}\label{L:TVV}
For each $q > \tfrac{d^2 + d + 2}2$
\begin{equation}\label{E:TVV}
\|\sum_{R \in \scriptR} \prod_{j = 1}^d \scriptE_\gamma f_{R, j}\|_{\frac qd}
\lesssim 
\Big(\sum_{\sigma \in S_d} \sum_{R \in \scriptR_{\sigma}} \|\prod_{j = 1}^d \scriptE_\gamma f_{R, j}\|_{\frac qd}^{(\frac qd)'}\bigr)^{\frac1{(\frac qd)'}}.
\end{equation}
\end{lemma}
\begin{proof}[Proof of Lemma~\ref{L:TVV}]
Let $A := (m, \vec n, k) \in \scriptA$. Recall that $\bar\gamma$ defined in \eqref{E:def bar gamma} obeys \eqref{E:moment plus err}. Let $\bar A := (K, \vec n, 0)$. By Lemma~\ref{L:convex approx}, there exists $\psi_{\bar A} \in \scriptS(\Rd)$ with $\psi_{\bar A} \equiv 1$ on $\bar \Gamma(R(\bar A))$, $\psi_{\bar A}$ supported on $\bar Q(\bar A)$, and $\|\hat \psi_{\bar A}\|_1 \sim 1$. Let $\psi_A(\xi) := \psi_{\bar A} \circ D_{2^{m - K}} T (k)^{-1}$. Then $\psi_A \equiv 1$ on $\Gamma(R(A))$, $\psi_A$ is supported on $Q(A)$, and $\|\hat \psi_A\|_1 \sim 1$. Moreover, by \eqref{E:finite overlap}, $\sum_{A \in \scriptA} \psi_A \sim 1$ on $\Gamma(Q \setminus \Delta)$.

Inequality \eqref{E:TVV} then follows from a well known interpolation argument of Tao-Vargas-Vega~\cite[Lemma~$6.1$]{TaoVargasVega98}.  
\end{proof}

The remainder of the proof of Proposition~\ref{P:positive bound extn} follows precisely that of \cite[Proposition $3.1$]{BiswasStovall20}. As that argument is rather long, we will not repeat it here. However, we briefly summarize.  For $p < d + 2$, we directly adapt the bilinear-to-linear argument of Tao--Vargas--Vega in~\cite{TaoVargasVega98}, and for $d + 2 \leq p < \frac{d^2 + d + 2}2$, it involves an adaptation~\cite[Lemma $3.8$]{BiswasStovall20} of the Marcinkiewicz interpolation argument in~\cite{SteinWeissbook} .
\end{proof}

The following result quantifies the heuristic that a few pieces are responsible for most of the $L^q$ norm of the extension of a near-extremizer.  

\begin{proposition}\label{P : uniform dyadic decay}
Let $q := \tfrac{d^2 + d}2 p' > \tfrac{d^2 + d + 2}2$ and $\gamma : \R \to \Rd$ be a polynomial of degree $N$. There exists $\rho_k = \rho_k (d, N, p) \searrow 0$ such that the following holds. For each $f \in L^p(\lambda_\gamma)$, there exists a sequence $\{\tau^k\}\subseteq \scriptI$, 
such that for $\{f^{> k}\}$ inductively defined by $f^{> 0} := f$ and 
\begin{equation}\label{partition:f}
f^k := f^{> k - 1} \rchi_{\tau^k \cap \{|f| < 2^k \lambda_\gamma(\tau^k)^{-1/p} \|f\|_{L^p(\lambda_\gamma)}\}}, \quad f^{> k} := f^{> k - 1} - f^k,
\end{equation}
for each positive integer $k$
$$
\|\scriptE_\gamma f^{> k}\|_q \leq \rho_k \|f\|_{L^p(\lambda_\gamma)}.  
$$
\end{proposition}
\begin{proof}
Multiplying by a constant if needed, we may assume that $\|f\|_{L^p(\lambda_\gamma)} = 1$. Employing the dominated convergence theorem, given $f^{> k - 1}$, we may select some $\tau^k \in \scriptI$ to maximize $\|f^k\|_{L^p(\lambda_\gamma)}$. With the sequence $\{f^k\}$ and $\{f^{> k}\}$ in~\eqref{partition:f}, for each positive integer $K$ let us denote
$$
R_K := \sup_{\tau \in \scriptI} \|f^{> 2 K} \rchi_{\tau \cap \{|f| < 2^K \lambda_\gamma(\tau)^{- \frac 1p}\}}\|_{L^p(\lambda_\gamma)}.
$$
By the maximality property of $\tau^k$ and the fact that the sequence $|f^{>k}|$ is (pointwise) decreasing, $R_K \leq \|f^{K + j}\|_{L^p(\lambda_\gamma)}$, for each $0 \leq j \leq K$. By mutual disjointness of $\{\supp(f^k)\}$, we have $K R_K^p \leq 1$. By Proposition~\ref{P:positive bound extn} this yields
$$
\|\scriptE_\gamma f^{> 2 K}\|_q \lesssim \max (2^{- c_p \theta K}, K^{- \frac{\theta}{p}}) \lesssim K^{- \frac{\theta}{p}}.
$$
This completes the proof with $\rho_k := C k^{- \frac{\theta}{p}}$.  
\end{proof}

Finally, we will need the following variant of Proposition~\ref{P:positive bound extn} in the next section.  

\begin{lemma}\label{L : large scriptE at a point}
Let $\gamma : \R \to \Rd$ be a polynomial of degree $N$. Then there exists $0 < \theta = \theta_d < 1$ such that
\begin{equation} \label{E:refined Drury 2}
\|\scriptE_\gamma f \|_{d^2 + d} \lesssim \sup_{\tau\in \scriptI} \Big(\lambda_\gamma(\tau)^{- \frac 12} \|\scriptE_\gamma (f \rchi_{\tau})\|_\infty \Big)^{1 - \theta} \|f\|_{L^2(\lambda_\gamma)}^\theta.
\end{equation}
\end{lemma}
\begin{proof}
Set $q_2 : = d^2 + d$, which is the exponent corresponding to $p = 2$ on the scaling line $q = \frac{d^2 + d}2 p'$.  Fix $q_1 < d^2 + d$ with $\theta := \frac{q_1}{d^2 + d}$ sufficiently close to $1$ for later purposes.  By Lemma~\ref{L:reduce dyadic} and the triangle inequality, it suffices to prove for $f$ supported on one of the $A_{j,k}$.  

From here, the proof is exactly the same as the proof of Proposition 3.10 of \cite{BiswasStovall20}, with the only change being incorporation of Lemma~\ref{L:TVV} of this article.
\end{proof}


\section{Spatial decomposition of frequency localized pieces} \label{S:profile}


In the previous section, we showed that most of the $L^q$ norm of the extension of an $L^p$ function $f$ comes from a few pieces, ``chips,'' of $f$ that are very well localized to intervals. However, these ``chips'' may still oscillate incoherently, leading to dispersion of their extensions.  In this section, we will decompose a sequence of chips into spatially coherent parts.  In a later section, we will go back and complete the frequency decomposition by proving that for an extremizing sequence, all chips have commensurable length scales and locations.  This represents a significant reorganization of the argument relative to previous works in this vein, such as \cite{BiswasStovall20, FlockStovall22, StovallParab}, and allows for a much simpler implementation of the next step by minimizing our reliance on pointwise geometric inequalities, which are not available when we need to consider interactions between different intervals $I_j$ in the decomposition \eqref{decompose R}.   

Applying the results of the previous section to a normalized extremizing sequence $\{f_n\}$ will produce sequences of localized chips $\{f_n^k\}$, adapted to sequences of intervals $\{\tau_n^k\}$.  After rescaling one such sequence $\{f_n^k\}$ and passing to a subsequence, we may assume that 1 is the left endpoint of $\tau_n^k$ for all $k$ and either $\tau_n^k$ is eventually (essentially) constant in $n$, or $\tau_n^k$ shrinks down to $\{1\}$ as $n \to \infty$.  In the former case, we will (eventually) deduce compactness of the original sequence modulo symmetries, while in the latter case, the curve is approximable by an affine copy of the moment curve, \eqref{E : local approx}.  Since analogous approximations are valid in the general polynomial case, we state and prove a more general result than what is needed immediately in the monomial case. 

\begin{proposition}\label{P : profiles}
Let $\gamma_L : \R \to \Rd$ be a polynomial of degree $N$ with $\lambda_{\gamma_L} \not \equiv 0$ and let  $\{\gamma_n : \R \to \Rd\}$ be a sequence of polynomials of degree at most $N$ with coefficients converging to those of $\gamma_L$.  Let $q = \frac{d^2 + d}2 p' > \frac{d^2 + d + 2}2$ and $R>0$. Let $\{f_n\}$ be a sequence of measurable functions with $|f_n| \leq R \rchi_{[- R, R]}$ for each $n$. Then, there exist sequences $\{\phi^j\}_{j \in \N} \subseteq  L^p(\lambda_{\gamma_L})$ and $\{x_n^j\}_{n,j \in \N} \subseteq \R^d$, such that, along a subsequence of $\{f_n\}$, the following hold with
$
r_n^J := f_n - \sum_{j = 1}^J e^{- i x_n^j \cdot \gamma_n} \phi^j$,  $J \in \N.
$
\begin{enumerate}[\rm(i)]
\item $e^{i x_n^j \cdot \gamma_n} f_n \rightharpoonup \phi^j$ weakly, as $n \to \infty$, $j \in \N$;
\item $\lim_n |x_n^j - x_n^{j'}| = \infty$ for each $j \neq j'$;
\item $\lim_n \|\scriptE_{\gamma_n} f_n\|_q^q - \sum_{j = 1}^J \|\scriptE_{\gamma_L} \phi^j\|_q^q - \|\scriptE_{\gamma_n} r_n^J\|_q^q = 0$ for each $J$;
\item $\sum_{j=1}^\infty \|\phi^j\|_{L^p(\lambda_{\gamma_L})}^{\tilde p} \leq \liminf_n \|f_n\|_{L^p(\lambda_{\gamma_n})}^{\tilde p}$, \quad where $\tilde p := \max (p, p')$;
\item $\lim_{J \to\infty} \limsup_n \|\scriptE_{\gamma_n} r_n^J\|_q = 0$.
\end{enumerate}
\end{proposition}

\begin{proof}
A key step in developing a nontrivial decomposition (i.e. satisfying conclusion (v)) will be to establish the existence of a non-zero weak limit for bounded $L^2$ sequences with non-negligible extensions. 

\begin{lemma}\label{L : positive weak lim}
Let the polynomials $\gamma_n, \gamma_L$ satisfy the hypotheses of Proposition~\ref{P : profiles} and let $\eps, D, R>0$. Consider a sequence $\{f_n\}$ of measurable functions such that for each $n$,
$$
|f_n| \leq R \rchi_{[- R, R]}, \quad \|f_n\|_{L^2(\lambda_{\gamma_n})} \leq D, \quad \|\scriptE_{\gamma_n} f_n\|_{d^2 + d} \geq \eps.
$$
Then, there exists a sequence $\{x_n \in \Rd \}$ so that along a subsequence, the modulated sequence $\{e^{i x_n \cdot \gamma_n} f_n\}$ converges weakly to some $\phi \in L^2(\lambda_{\gamma_L})$ with
\begin{equation}\label{positive lower bound for weak lim}
\|\phi\|_{L^2(\lambda_{\gamma_L})} \gtrsim \eps \bigl(\tfrac\eps D\bigr)^C,
\end{equation}
with $C:=\tfrac{\theta}{1-\theta}$, $\theta$ as in Lemma~\ref{L : large scriptE at a point}.
\end{lemma}

Conversely, we note that if $|f_n| \leq R\rchi_{[-R,R]}$ and $\scriptE_{\gamma_n} f_n \to 0$ in $L^{d^2+d}$, then $f_n \rightharpoonup 0$ weakly, by basic distributional considerations.  

\begin{proof}
By our hypotheses $\lambda_{\gamma_n} \to \lambda_{\gamma_L}$ uniformly $[- R, R]$. Applying Lemma~\ref{L : large scriptE at a point} to each polynomial $\gamma_n$, there exist $0 < \theta < 1$ and $\tau_n \subset [- R, R]$ so that for each $n$
\begin{equation}\label{large dyadic interval for uniform approximation: non zero limit}
\eps\bigl(\tfrac{\eps}D\bigr)^C \lesssim \lambda_{\gamma_n} (\tau_n)^{- \frac 12} \|\scriptE_{\gamma_n} (f_n\rchi_{\tau_n})\|_\infty.
\end{equation}

By our hypothesized bounds on $f_n$ and H\"older's inequality, $|\tau_n|$ is bounded below.  By Heine-Borel, after passing to a subsequence, there exists a fixed interval $\tau \subseteq [-R,R]$ of positive length such that $|\tau_n \Delta \tau| \to 0$, where $\Delta$ denotes the symmetric difference of sets.  
Applying H\"older's inequality again and $\gamma_n \to \gamma_L$, we obtain from \eqref{large dyadic interval for uniform approximation: non zero limit} for $n$ sufficiently large that 
$$
\lambda_{\gamma_L}(\tau)^{1/2}\eps\bigl(\tfrac{\eps}D\bigr)^C \lesssim  \|\scriptE_{\gamma_n} (f_n\rchi_{\tau})\|_\infty.
$$
Thus there exists $\{x_n\} \in \Rd$ such that for $n$ sufficiently large
\begin{equation}\label{large scriptE value at 0}
|\scriptE_{\gamma_n} (e^{i x_n \cdot \gamma_n}f_n\rchi_\tau) (0)| = |\scriptE_{\gamma_n} (f_n\rchi_\tau) (x_n)| \gtrsim \lambda_{\gamma_L}(\tau)^{\frac 12} \eps\bigl(\tfrac\eps D\bigr)^C. 
\end{equation}

Being $L^2$-bounded, along a subsequence, $e^{i x_n \cdot \gamma_n} f_n \rightharpoonup \phi$ weakly. Since $|f_n| \leq R \rchi_{[- R, R]}$, $\gamma_n \to \gamma_L$, and the triangle inequality,
\begin{align*}
|\scriptE_{\gamma_L}(\phi)_\tau (0)| = \lim_n |\scriptE_{\gamma_n} (e^{i x_n \cdot \gamma_n}f_n\rchi_\tau) (0)| \gtrsim \lambda_{\gamma_L}(\tau)^{\frac 12} \eps \bigl(\tfrac\eps D\bigr)^C. 
\end{align*}
On the other hand, by H\"older's inequality
$$
|\scriptE_{\gamma_L}(\phi \rchi_\tau)(0)| \leq \|\phi\rchi_\tau\|_{L^1(\lambda_{\gamma_L})} \leq \lambda_{\gamma_L}(\tau)^{1/2}\|\phi\|_{L^2(\lambda_{\gamma_L})},
$$
and the lower bound in \eqref{positive lower bound for weak lim} follows.  
\end{proof}

Following the outline from \cite{StovallParab}, we will first prove Proposition~\ref{P : profiles} in the case $p=2,q=q_2$, wherein we can take advantage of the Hilbert space structure of $L^2(\lambda_\gamma)$, and then we will use that case to deduce the result for general $p,q$.  

\textbf{Proof of Proposition~\ref{P : profiles} for $p = 2$:}  For $p = 2$ we will prove something slightly stronger, replacing (iv) by the improvement,
\begin{equation}\label{E : L^2 decoup}
\lim_n \|f_n\|_{L^2(\lambda_{\gamma_n})}^2 - \sum_{j = 1}^J \|\phi^j\|_{L^2(\lambda_{\gamma_L})}^2 -  \|r_n^J\|_{L^2(\lambda_{\gamma_n})}^2 = 0, \quad J \in\N.
\end{equation}

We set  $D_0:=\limsup_n \|f_n\|_{L^2(\lambda_{\gamma_n})}$ and $\eps_0:=\limsup_n \|\scriptE_{\gamma_n} f_n\|_{q_2}$, and note that $\eps_0 \lesssim D_0$ by Theorem~\ref{T:extn polynomial} and $D_0 < \infty$ by H\"older. 
Along a subsequence, $\|f_n\|_{L^2(\lambda_\gamma)} \lesssim D_0$ and $\|\scriptE_{\gamma_n} f_n\|_{q_2} \gtrsim \eps_0$ for all sufficiently large $n$. 

If $\eps_0 > 0$, we apply Lemma~\ref{L : positive weak lim} to extract a profile $\phi^1 \in L^2(\lambda_{\gamma_L})$ with
$$
\|\phi^1\|_{L^2(\lambda_{\gamma_L})} \gtrsim \eps_0\bigl(\tfrac{\eps_0}{D_0}\bigr)^C,
$$
together with $\{x_n^1 \in \Rd\}$ so that $e^{i x_n^1 \cdot \gamma_n} f_n \rightharpoonup \phi^{1}$ weakly. If $\eps_0 = 0$, we set $\phi^1 \equiv 0$ and choose $\{x_n^1\}$ arbitrarily, noting the remark after Lemma~\ref{L : positive weak lim}. Irrespective of the value of $\eps_0$,
$$
\lim_n \|f_n\|_{L^2(\lambda_{\gamma_L})}^2 - \|\phi^1\|_{L^2(\lambda_{\gamma_L})}^2 - \|f_n - e^{- i x_n^1 \cdot \gamma_n} \phi^1\|_{L^2(\lambda_{\gamma_L})}^2 = 0.
$$ 
As $|f_n| \leq R \rchi_{[- R, R]}$ and $\gamma_n \to \gamma_L$ locally uniformly, this implies that
\begin{align}\label{L^2 decoupling for weak limits}
\lim_n \|f_n\|_{L^2(\lambda_{\gamma_n})}^2 - \|\phi^1\|_{L^2(\lambda_{\gamma_L})}^2 - \|f_n - e^{i x_n^1 \cdot \gamma_n} \phi^1\|_{L^2(\lambda_{\gamma_n})}^2 = 0.
\end{align}
In other words, \eqref{E : L^2 decoup}, our improved version of (iv) holds when $J=1$.  That conclusions (i, ii) hold for $J = 1$ is immediate. To prove that conclusion (iii) holds for $J = 1$, we decompose
\begin{align}\label{expand and apply Brezis-Lieb}
\scriptE_{\gamma_n} (e^{i x_n^1 \cdot \gamma_n} f_n) = \scriptE_{\gamma_L} \phi^1 + \scriptE_{\gamma_n} (e^{i x_n^1 \cdot \gamma_n} f_n - \phi^1) + \Big(\scriptE_{\gamma_n} \phi^1 - \scriptE_{\gamma_L} \phi^1\Big).     
\end{align}
The third term in above RHS goes to zero in $L^{q_2}$ by Lemma~\ref{L:op convergence} as $\gamma_n \to \gamma_L$ locally uniformly. The second term converges to zero pointwise by the uniform compact supports of the $f_n$, uniform convergence of $\gamma_n \to \gamma_L$ on those supports, and weak convergence of $e^{ix_n^1 \cdot \gamma_n} f_n$ to $\phi^1$.  Since $\scriptE_{\gamma_n} (e^{i x_n^1 \cdot \gamma_n} f_n) = \scriptE_{\gamma_n} f_n (\cdot + x_n^1)$, an application of the generalized Brezis-Lieb lemma~\cite[Lemma $3.1$]{Frank_etl_2016} yields
\begin{equation}\label{apply Brezis-Lieb for cauchy}
\lim_n \|\scriptE_{\gamma_n} f_n\|_{q_2}^{q_2} - \|\scriptE_{\gamma_L} \phi^1\|_{q_2}^{q_2} - \|\scriptE_{\gamma_n} r_n^1\|_{q_2}^{q_2} = 0.  
\end{equation}
Thus, conclusion (iii) holds for $J = 1$.

To complete the profile decomposition, we use induction and Cantor's diagonalization argument (to ensure a nonempty subsequence).  With $r_n^0:=f_n$, at each stage, we apply the preceding to $r_n^J:=r_n^{J-1} -e^{-ix_n^J\cdot \gamma_n}\phi^J$ to extract a subsequence,  $\phi^J$, and $\{x_n^J\}$.  If we ever have $\eps_J:=\limsup_n \|\scriptE_{\gamma_n} r_n^J\|_{q_2}=0$, then we set $\phi^J \equiv 0$ and choose $\{x_n^J\}$ arbitrarily, subject to the condition that conclusion (ii) must hold relative to previously selected $\{x_n^j\}$.

We claim that the conclusion of the proposition holds with these sequences $\{\phi^j\}, \{x_n^j\}$. Conclusion (iii) and the improved version of conclusion (iv), \eqref{E : L^2 decoup}, hold by induction.  

Suppose that conclusion (ii) fails.  We fix the least index $j$ for which $x_n^j-x_n^{j'}$ converges along a subsequence for some $j'<j$ and let $x^{jj'}$ denote the limit.  Since $j$ was least, for all other pairs $k \neq k' \leq j$, $|x_n^k-x_n^{k'}| \to \infty$ along this subsequence.  In particular, by stationary phase and $\gamma_n \to \gamma_L$ locally uniformly, multiplication by $e^{i(x_n^j-x_n^{j'})\cdot \gamma_n}$ converges to multiplication by $e^{ix_n^{jj'}\cdot \gamma_L}$ in the strong operator topology and multiplication by $e^{i(x_n^j-x_n^k)\cdot \gamma_n}$ converges to zero in the weak operator topology for all $j'\neq k < j$.  Therefore, 
\begin{align*}
\phi^j
& = \wklim e^{i(x_n^j-x_n^{j'})\cdot \gamma_n}(e^{ix_n^{j'}\cdot \gamma_n}r_n^{j'}-\phi^{j'} + \sum_{k=j'+1}^{j-1}e^{i(x_n^{j}-x_n^k)\cdot \gamma_n}\phi^k)=0.
\end{align*}
But for $\phi^j=0$, we chose $\{x_n^j\}$ so that conclusion (ii) would hold, a contradiction.   

Similar arguments give conclusion (i).

Conclusion (v) follows from conclusion (iii) and \eqref{E : L^2 decoup}. Indeed, both 
$$ 
D_J:=\limsup_n \|r_n^J\|_{L^2(\lambda_{\gamma_n})}\qtq{and} \eps_J:=\limsup_n \|\scriptE_{\gamma_n} r_n^J\|_{q_2}
$$ are decreasing in $J$, and hence both have a limit.  If (v) fails, $\lim_J \eps_J =\eps>0$, yielding a contradiction by \eqref{positive lower bound for weak lim}.

{\bf Proof for general p :} For  $1 < p < \frac{d^2 + d + 2}2$ and $p \neq 2$,  we let $\{\phi^j\}$ and $\{x_n^j\}$ be the sequences produced in the $p = 2$ case above. We claim that conclusions (i) - (v) hold with these same profiles $\{\phi^j\}$ and sequences $\{x_n^j\}$ for $L^p(\lambda_\gamma)$. 

Conclusions (i) and (ii) follow directly from our construction of the profiles. Application of the generalized Br{\'e}zis-Lieb lemma~\cite[Lemma $3.1$]{Frank_etl_2016} gives conclusion (iii) by identical reasoning to the case $q = q_2$. We choose $\frac{d^2 + d + 2}2 < q_1 < \infty$ so that $q$ lies between $q_1$ and $q_2$. Since conclusion (iii) holds for $q_1$, we have that $\limsup_n \|\scriptE_{\gamma_n} r_n^J\|_{q_1}$ is uniformly bounded in $J$.  (The bound depends on $R$.) With $0 < \theta < 1$ defined by $\frac 1q =: \frac{1 -\theta}{q_1} + \frac {\theta}{q_2}$, by H\"older's inequality
$$
\limsup_J \limsup_n \|\scriptE_{\gamma_n} r_n^J\|_q \lesssim_{q_1, R,\gamma_L} \limsup_J \limsup_n \|\scriptE_{\gamma_n} r_n^J\|_{q_2}^\theta = 0. 
$$ 
This establishes conclusion (v) for $q$.

Finally we are left to prove conclusion (iv). Let us fix $J\in\N$. We let 
\begin{equation}
g_n : = f_n \lambda_{\gamma_n}^{\frac 1p} \in L^p(\R), \quad \psi^j := \phi^j \lambda_{\gamma_L}^{\frac 1p} \in L^p(\R), \quad 1 \leq j \leq J.    
\end{equation}
Since $|f_n| \leq R \rchi_{[- R, R]}$ and $\lambda_{\gamma_n} \to \lambda_{\gamma_L}$ locally uniformly, $e^{i x_n^j \cdot \gamma_n} g_n \rightharpoonup \psi^j$ weakly. Fixing a positive $\eps$, we choose even non-negative $a : \R \to \R$ with $a(0) = \int a = 1$, and $b : \R \to \R$ with $|b| \leq 1$, both compactly supported and smooth, so that 
\begin{equation}\label{choose test function b}
\|\psi^j - a \ast (b \psi^j)\|_p < \eps, \qtq{for each} 1 \leq j \leq J.
\end{equation}
For $\pi_n^j f := a \ast (b e^{i x_n^j \cdot \gamma_n} f)$, using the weak convergence $e^{i x_n^j \cdot \gamma_n} g_n \rightharpoonup \psi^j$ and compact supports of $a, b, g_n$, by the dominated convergence theorem,
$$
\lim_n \|\pi_n^j g_n - a \ast (b \psi^j)\|_p = 0, \qtq{for each} 1 \leq j \leq J.
$$
Let $P_n^J := \{\pi_n^j\}_{j = 1}^J$. Letting $\eps \to 0$, it suffices to show that
\begin{equation*}
\limsup_n \|P_n^J\|_{L^p \to l^{\tilde p}(L^p)} \leq 1, \quad 1 < p < \frac{d^2 + d +2}2.
\end{equation*}
Since the above holds trivially for $p = 1, \infty$, by complex interpolation~\cite{Stein56} and duality, it suffices to prove that $\limsup_n \|(P_n^J)^\ast\|_{l^2(L^2) \to L^2} \leq 1$. For $F = \{f_j\}_{j = 1}^J$, 
$$
\|({P_n^J})^\ast F\|_2^2 \leq \sum_{j = 1}^J \|f_j\|_2^2 + \sum_{1 \leq j \neq j' \leq J} \langle \pi_n^j (\pi_n^{j'})^\ast f_{j'}, f_j \rangle,
$$
by Young's convolution inequality. Consequently, it suffices to show that 
\begin{equation}\label{pi pi'^* goes to zero}
\lim _n \|\pi_n^j (\pi_n^{j'})^\ast\|_{L^2 \to L^2} = 0, \quad j \neq j'.
\end{equation}

Note that $(\pi_n^j)^\ast \phi := b e^{i x_n^j \cdot \gamma_n} (a \ast \phi)$ and so
$$
\pi_n^j (\pi_n^{j'})^\ast f (t) = \int \Phi_n^{j, j'} (t, s) f(s) \,\, ds, \quad t \in \R,
$$
where the kernel $\Phi_n^{j, j'}$ is given by
$$
\Phi_n^{j, j'} (t, s) := \int a(t - u) e^{i (x_n^{j'} - x_n^j) \cdot \gamma_n (u)} b(u)^2 a(u - s) \,\, du.
$$
Since $\lambda_{\gamma_L} \not \equiv 0$, by stationary phase using conclusion (ii), uniform (in $n$) compact supports of $\Phi_n^{j, j'}$ and the local uniform convergence $\gamma_n \to \gamma_L$,
$$
\lim_n \sup_t \|\Phi_n^{j, j'} (s, t)\|_{L_s^1} = 0, \qtq{and} \lim_n \sup_s \|\Phi_n^{j, j'} (s, t)\|_{L_t^1} = 0.
$$
Applying Schur's test implies~\eqref{pi pi'^* goes to zero}. The proof of Proposition~\ref{P : profiles} is now complete.
\end{proof}


\section{Uniform frequency localization} \label{S:chunks}


In this section, we resume the frequency localization initiated in Proposition~\ref{P : uniform dyadic decay} by proving that when $f$ is a near extremizer, all of the intervals $\tau^k$ making significant contribution to $\scriptE_{\gamma} f$ are compatible with one another.  More precisely, we will prove the following. 

\begin{proposition} \label{P:freq loc}
Let $\{f_n\}$ be an $L^p(\lambda_\gamma)$-normalized extremizing sequence.  Then, along a subsequence, there exist a sequence $\{S_n\}$ of dilations such that either
\begin{equation}\label{E:cpct freq loc}
\lim_{R \to \infty} \lim_{n \to \infty} \|(S_n f_n) \rchi_{\{|t|>R\}} \rchi_{\{|S_nf_n|>R\}}\|_{L^p(\lambda_\gamma)} = 0    
\end{equation}
or
\begin{equation}\label{E:conc freq loc}
    \lim_{R \to \infty} \lim_{n \to \infty} \|(S_nf_n)\rchi_{\{\dist(t,[1])>R\delta_n\}}\rchi_{\{|S_nf_n|>R\delta_n^{-1/p}\}}\|_{L^p(\lambda_\gamma)} = 0,
\end{equation}
for some sequence $\delta_n \searrow 0$. Here, $[1] := 
\{1,-1\}$, if $\gamma$ is even or odd, and $\{1\}$, otherwise.
\end{proposition}

As we will see, the advantage of already having the decomposition in Proposition~\ref{P : profiles} is that it will allow us to prove certain orthogonality estimates without use of the geometric inequality \eqref{E:geom ineq}, which may fail when the $t_i$ are not all contained in the same interval $I_j$ from \eqref{decompose R}. 

\begin{proof}
We apply Proposition~\ref{P : uniform dyadic decay} to each $f_n$, and let $\tau_n^k$ denote the resulting intervals and $f_n^k$ the corresponding functions.  By construction and our assumption that $\{f_n\}$ is extremizing, $\liminf \|f_n^1\|_p > 0$. By rescaling as in~\eqref{E : def scaling} (potentially by a negative number) and passing to a subsequence, we may assume that $\tau_n^1 = 1+[0,\eps_n^1]$ and that either $\eps_n^1$ converges to some $0<\eps^1<\infty$, or  $\eps_n^1 \searrow 0$, as $n \to \infty$.

In fact, in the former case, we may assume that $\tau_n^1 \equiv \tau^1$ for all sufficiently large $n$. Indeed, $|\tau_n^1 \Delta \tau^1| \to 0$, where $\tau^1 := 1 + [0,\eps^1]$, so by H\"older's inequality, $\|\scriptE_{\gamma}(f_n^1-\tilde f_n^1)\|_q \to 0$, where $\tilde f_n^1$ is constructed analogously to $f_n^1$, but with $\tau^1$ in place of $\tau_n^1$. 

We will prove that \eqref{E:cpct freq loc} occurs when $\tau_n^1$ is eventually constant and that \eqref{E:conc freq loc} occurs when $\eps_n^1 \searrow 0$.  Thanks to the uniform decay in Proposition~\ref{P : uniform dyadic decay}, it suffices to prove that, along a subsequence, each $\{\tau_n^k\}$, $\tau_n^k = [a_n^k,a_n^k(1+\eps_n^k)]$, $\eps_n^k \lesssim 1$,  
arising in subsequent stages of the decomposition is either negligible, i.e., $\lim_{n \to \infty} \|f_n^k\|_{L^p(\lambda_\gamma)} = 0$, or is commensurable with $\{\tau_n^1\}$ in the sense that the following all hold as $n \to \infty$: 
\begin{equation}\label{E:good}
\begin{gathered} 
\lim_{n \to \infty}  (\eps_n^1)^{-1}\lambda_\gamma(\tau_n^k) \in (0,\infty), \qquad 
\lim_{n \to \infty} a_n^k \in \R,\\ 
 \lim_{n \to \infty} (\eps_n^1)^{-1} \dist([1],a_n^k) \in [0,\infty).
\end{gathered}
\end{equation}
(We may assume that these, and all other limits arising in the proof, exist by passing to a subsequence.)

The preceding reduction leads us to say that an index $k$ is good if the conditions in \eqref{E:good} all hold and bad otherwise.  In other words, we want to show that every bad $k$ is negligible (in the sense that $\{\tau_n^k\}$ is negligible).  

\begin{lemma}\label{L:no bad k}
If $k$ is bad and $k'$ is good, 
\begin{equation} \label{E:no bad k}
\lim_{n \to \infty} \|(\scriptE_\gamma f_n^k)(\scriptE_\gamma f_n^{k'})\|_{q/2} = 0.
\end{equation}
\end{lemma}

As shown in (e.g.) the proof of \cite[Proposition 4.1]{BiswasStovall20}, by convexity ($p<q$) and basic estimates, we may now conclude the proof of Proposition~\ref{P:freq loc}, modulo the proof of Lemma~\ref{L:no bad k}.    
\end{proof}

\begin{proof}[Proof of Lemma~\ref{L:no bad k}]
We begin with several reductions. First, we observe that if $\tau_n^{k'}$ is good, then rescaling so that $\tau_n^{k'}=1+[0,\eps_n^{k'}]$ and then interchanging the indices 1 and $k'$ does not change the good/bad intervals. Thus it suffices to prove the lemma in the case $k'=1$.  

Given $k$, we consider the sequence of polynomials 
$$
\gamma_n^k:= D_{(\eps_n^k)^{-1}} T(1)^{-1}[\gamma(\eps_n^k(s-1)+1)-\gamma(1)].
$$
After a change of variables, some algebraic manipulations, and recalling the construction in Proposition~\ref{P : profiles}, we see that
\begin{equation} \label{E:Egammank}
\scriptE_\gamma f_n^k(x) = |a_n^k|^{\frac{2|\vec l|}{d(d+1)p'}}(\eps_n^k)^{1/p'}\exp(iD_{a_n^k}x \cdot \gamma(1))\scriptE_{\gamma_n^k}\tilde f_n^k(D_{\eps_n^k} T(1)^t D_{a_n^k}^{\vec l} x),
\end{equation}
where 
$$
\tilde f_n^k(s):=\lambda_\gamma(\tau_n^k)^{1/p}f_n^k(a_n^k(\eps_n^k(s-1)+1))
$$
obeys the bound
$|\tilde f_n^k| \lesssim 2^k \rchi_{[1,2]}$.  (We note that $\lambda_\gamma(\tau_n^k) \sim |a_n^k|^{\frac{2|\vec l|}{d(d+1)}}\eps_n^k$.)

By Taylor's theorem, $\gamma_n^k \to \gamma_0(\cdot-1)$.  Thus, by Proposition~\ref{P : profiles}, specifically conclusion (v), $\scriptE_{\gamma_n^k}\tilde f_n^k$ can be approximated in $L^q$ to arbitrary precision by a finite sum of the form
$$
\sum_{j=1}^J \scriptE_{\gamma_n^k} (e^{-ix_n^{k,j}\cdot \gamma_n^k}\phi^{k,j}).
$$
The arguments above of course also apply with $k=1$, so by the triangle inequality and H\"older's, the lemma follows once we prove \eqref{E:no bad k} in the case $k$ bad, $k'=1$, and
$$
\tilde f_n^1 = e^{-ix_n^1 \cdot \gamma_n^1}\phi^1, \qquad \tilde f_n^k = e^{-ix_n^k \cdot \gamma_n^k}\phi^k.
$$

By Lemma~\ref{L:op convergence}, with $G^k:=\scriptE_{\gamma_0(\cdot-1)} \phi^k$
$$
\|\scriptE_{\gamma_n^k}(e^{-ix_n^k \cdot \gamma_n^k}\phi^k)(x) - G_k(x-x_n^k)\|_{L^q_x} \to 0,
$$
Yet another application of the triangle inequality allows us to replace $\scriptE_{\gamma_n^k}\tilde f_n^k$ with $G^k(\cdot-x_n^k)$ in \eqref{E:Egammank}.  (Similarly with $k \rightsquigarrow 1$.) We look to \eqref{E:Egammank} for help unwrapping all of these reductions.  Recalling that $a_n^1=1$ and $\lambda_\gamma(\tau_n^k) = |a_n^k|^{\frac{2|\vec l|}{d(d+1)}}\eps_n^k$, it remains to prove that
$$
\lim_{n \to \infty} \|(\eps_n^1)^{1/p'}G^1(D_{\eps_n^1}T(1)^t x - x_n^1) \lambda_\gamma(\tau_n^k)^{1/p'}G^k(D_{\eps_n^k} T(1)^t D_{a_n^k}^{\vec l} x-x_n^k)\|_{q/2} = 0,  
$$
for arbitrary bad $k$, $G^1,G^k \in L^q$ and $\{x_n^1\},\{x_n^k\} \subseteq \R^d$. By standard density arguments and positivity, we may in fact assume that $G^1$ and $G^k$ equal the characteristic function of some $B_R(0)$. By rescaling in $x$, we may further assume that $R=1$.   

We have thus reduced matters to bounding the volumes of intersections of pairs of ellipsoids (times some scaling factor). Such volumes are maximized when both ellipsoids are centered at 0, so we may assume $x_n^1 \equiv x_n^k \equiv 0$. Next, recalling that $(\det D_\eps)^{2/q} = \eps^{2/p'}$, we may change variables so that all symmetries act on $G^k$.  

Our reductions are finally complete;  by the identity $T(t) = D_t^{\vec l} T(1) D_{t^{- 1}}$, it remains to prove that whenever $k$ is bad,
\begin{equation} \label{E:no bad k simple}
\lim_{n \to \infty} \left\|\bigl(\tfrac{\lambda_\gamma(\tau_n^k)}{\eps_n^1}\bigr)^{1/p'}\rchi_{B_1(0)}(D_{a_n^k\eps_n^k}T(a_n^k)^tT(1)^{-t}D_{(\eps_n^1)^{-1}} y)\right\|_{L^{q/2}_y(B_1(0))} = 0.
\end{equation}

We now refer back to \eqref{E:good}. If $k$ is bad because  
$
\lim_{n \to \infty}{\lambda_\gamma(\tau_n^k)}/{\eps_n^1} = 0,
$ 
i.e., $\tau_n^k$ is too short, 
\eqref{E:no bad k simple} follows from H\"older's inequality.  The roles of 1 and $k$ are essentially symmetric (1 is bad with respect to $k$), so we similarly conclude that \eqref{E:no bad k simple} holds when $\tau_n^k$ is too long.  We may henceforth assume that $\tau_n^k$ is neither too short nor too long, and hence may ignore the multiplicative factor in front of $\rchi_{B_1(0)}$ in \eqref{E:no bad k simple}.  

Suppose that $a_n^k \to \pm\infty$.  We consider (the column vector)
$$
\omega_n^k:=(\gamma_d'(1),\eps_n^1\gamma_d''(1),\ldots,(\eps_n^1)^d\gamma^{(d)}_d(1)),
$$
where we are differentiating the $d$-th component of $\gamma$. Then  $|\omega_n^k|\sim 1$ and 
$$
D_{a_n^k\eps_n^k}T(a_n^k)^tT(1)^{-t}D_{(\eps_n^1)^{-1}} \omega_n^k = (\eps_n^1)^{-1} D_{a_n^k\eps_n^k} T(a_n^k)^t e_d,
$$
which has norm about $\tfrac{\eps_n^k}{\eps_n^1}(a_n^k)^{l_d}$.  Thus the linear transformation in the argument of $\rchi_{B_1(0)}$ in \eqref{E:no bad k simple} has norm at least (a constant times) $|a_n^k|^{l_d}\eps_n^k/\eps_n^1$.  On the other hand, since $\tau_n^k$ is not too short, 
$
(a_n^k)^{{2|\vec l|}/{d(d+1)}} \eps_n^k/\eps_n^1 
$ 
is bounded below.  Since $|a_n^k| \to \infty$ and 
$$
l_d > \tfrac{2l_d}{d+1} > \tfrac{2|\vec l|}{d(d+1)},
$$
we may conclude that $|a_n^k|^{l_d}\eps_n^k/ \eps_n^1 \to \infty$ as $n \to \infty$. Using Fubini to integrate \eqref{E:no bad k simple} along eigenspaces of the linear operator, the left hand side of \eqref{E:no bad k simple} (recall that $\tfrac{\lambda_\gamma(\tau_n^k)}{\eps_n^1} \sim_k 1$) is bounded by a dimensional constant times
$$
\lim_{n \to \infty} \|D_{a_n^k\eps_n^k}T(a_n^k)^tT(1)^{-t}D_{(\eps_n^1)^{-1}}\|^{-1} = 0.
$$
As in the case when $\tau_n^k$ was too long or short, we may obtain the same conclusion when $|a_n^k| \to 0$.  

We may henceforth assume that $\{a_n^k\}$ converges in $\R \setminus \{0\}$. As $\tau_n^k$ is neither too long nor too short, $\{\eps_n^k/\eps_n^1\}$ converges to a limit in $(0,\infty)$.  Referring back to \eqref{E:good}, we have run out of ways for $k$ to be bad if $\eps_n^1$ is eventually constant, so we may assume that $\eps_n^1 \to 0$.  By taking a slightly larger ball and then dilating, we may neglect compact sequences of dilations in \eqref{E:no bad k simple}, and it remains to prove that
\begin{equation} \label{E:no bad k simpler}
\lim_{n \to \infty} \bigl\|\rchi_{B_1(0)}(D_{\eps_n^1} T(a_n^k)^tT(1)^{-t}D_{(\eps_n^1)^{-1}} y)\bigr\|_{L^{q/2}_y(B_1(0))} = 0 \end{equation}
in the case that 
$$
\lim_{n \to \infty} (\eps_n^1)^{-1}\dist(a_n^k, [1]) = \infty.  
$$

As we've already noted, it suffices to prove that 
$$
\lim_{n \to \infty} \bigl\| D_{\eps_n^1}T(a_n^k)^tT(1)^{-t}D_{(\eps_n^1)^{-1}}   \bigr\| = \infty.  
$$
Using standard matrix manipulations and the fact that $\|T(1)\|,\|T(1)^{-1}\| \sim 1$, 
\begin{align*}
&\bigl\| D_{\eps_n^1}T(a_n^k)^tT(1)^{-t}D_{(\eps_n^1)^{-1}}   \bigr\|
 \sim \max_{i,j} \bigl|\bigl[D_{(\eps_n^1)^{-1}} T(1)^{-1}T(a_n^k)D_{\eps_n^1}\bigr]_{ij}\bigr| \\
 &\qquad
  = \max_{i,j} (\eps_n^1)^{j - i}\bigl|\bigr[T(1)^{-1}T(a_n^k)\bigr]_{ij}\bigr| \geq \max_{2 \leq i \leq d} (\eps_n^1)^{-1}\bigl|\bigl[T(1)^{-1}T(a_n^k)\bigr]_{i1}\bigr| \\
  &\qquad \gtrsim(\eps_n^1)^{-1}\bigl|\gamma'(a_n^k) \wedge \gamma'(1)\bigr| \sim (\eps_n^1)^{-1} \max_{i < j}|a_n^k|^{n_i-1}|(a_n^k)^{n_j-n_i}-1|.
\end{align*}
Thus, taking the limit, with $a := \lim a_n^k \in \R \setminus\{0\}$,
$$
\lim_{n \to \infty} \bigl\| D_{\eps_n^1}T(a_n^k)^tT(1)^{-t}D_{(\eps_n^1)^{-1}}   \bigr\| \gtrsim_a \lim_{n \to \infty} (\eps_n^1)^{-1} \dist(a_n^k,[1]) \to \infty.
$$
\end{proof}


\section{$L^p$ convergence to an extremizer in the case of non-concentration} \label{S:Lp convergence}


Proposition~\ref{P:freq loc} left us with two possible outcomes for an $L^p(\lambda_\gamma)$-normalized extremizing sequence, namely, concentration \eqref{E:conc freq loc}, or nonconcentration \eqref{E:cpct freq loc}.  In this section, we address the latter case, revisiting the profile decomposition \ref{P : profiles} to prove convergence to an extremizer.  More precisely, we will prove the following.

\begin{proposition} \label{P:convergence}
Any normalized extremizing sequence obeying \eqref{E:cpct freq loc} possesses a subsequence that converges in $L^p(\lambda_\gamma)$ modulo symmetries.   
\end{proposition}

\begin{proof}
Let $\{f_n\}$ be a $L^p(\lambda_\gamma)$-normalized extremizing sequence of $\scriptE_\gamma : L^p(\lambda_\gamma) \to L^q$ that obeys \eqref{E:cpct freq loc}.  Incorporating the dilations into the $f_n$, we may assume that each $S_n$ equals the identity.  For each positive integer $R$, we define $f_n^R := f_n \, \rchi_{[- R, R] \cap \{|f_n| \leq R\}}$.  By the triangle inequality,
$$
\liminf_n \|\scriptE_\gamma f_n^R\|_q = B_{\gamma, p} - o_R(1), 
$$
where $o_R(1) \to 0$ as $R \to \infty$.  For each fixed $R$, applying Proposition~\ref{P : profiles} to the truncated sequence $\{f_n^R\}$ with $\gamma_n \equiv \gamma = \gamma_L$ produces sequences $\{x_n^{j, R} \in \Rd\}$ and profiles $\{\phi_j^R \in L^p(\lambda_\gamma)\}$ so that the following holds. 
\begin{align*}
B_{\gamma, p}^q - o_R(1) & = \liminf_n \|\scriptE_\gamma f_n^R\|_q^q \leq \sum_{j = 1}^{\infty} \|\scriptE_\gamma \phi_j^R\|_q^q \leq B_{\gamma, p}^q \, \sum_{j = 1}^{\infty} \|\phi_j^R\|_{L^p(\lambda_\gamma)}^q
\\
&  \leq B_{\gamma, p}^q \sup_j \|\phi_j^R\|_{L^p(\lambda_\gamma)}^{q - \tilde p} \, \sum_{j = 1}^{\infty} \|\phi_j^R\|_{L^p(\lambda_\gamma)}^{\tilde p} \leq B_{\gamma, p}^q \, \sup_j \|\phi_j^R\|_{L^p(\lambda_\gamma)}^{q - \tilde p}.  
\end{align*}
Combining this with conclusion (iv) in Proposition~\ref{P : profiles} implies that for each $R$ sufficiently large, there exists a unique profile $\phi_j^R$ with $\|\phi_j^R\|_{L^p(\lambda_\gamma)} = 1 - o_R(1)$. Henceforth, we omit the index $j$ from both $\phi^R_j$ and $x^R_j$. By the strict convexity of $L^p(\lambda_\gamma)$ (see also the proof of Lemma $2.11$ in \cite{LL}) and since $e^{i x_n^R \cdot \gamma} f_n^R \rightharpoonup \phi^R$ weakly
$$
\limsup_n \|f_n^R - e^{- i x_n^R \cdot \gamma} \phi^R\|_{L^p(\lambda_\gamma)} = o_R(1).
$$
By Proposition~\ref{P:freq loc} and the triangle inequality 
\begin{equation}\label{f_n close to profile in L^p} 
\limsup_n \|f_n - e^{- i x_n^R \cdot \gamma} \phi^R\|_{L^p(\lambda_\gamma)} = o_R(1),	
\end{equation}
and so, another application of the triangle inequality produces
\begin{equation}\label{e^{ix^Rcdot gamma} phi^R is cauchy}
\sup_{R_1,R_2 \geq R}\limsup_n \| e^{- i x_n^{R_1} \cdot \gamma} \phi^{R_1} - e^{- i x_n^{R_2} \cdot \gamma} \phi^{R_2}\|_{L^p(\lambda_\gamma)} = o_R(1).
\end{equation}

Next, we show that along a further subsequence $\{\phi^R\}$ converges (up to modulations) in $L^p(\lambda_\gamma)$ as $R \to \infty$. To this end, first we observe that along a subsequence, for each fixed large enough $R$ and $R'\geq R$, the sequence $\{x_n^R-x_n^{R'}\}$ converges as $n \to \infty$. To see this, suppose, to the contrary, that  $\|x_n^R - x_n^{R'}\| \to \infty$ for some fixed large $R$, and $R' > R$. Thus, by stationary phase, multiplication by $e^{i (x_n^R - x_n^{R'}) \cdot \gamma}$ converges to zero in the weak operator topology. Consequently, by H\"older's inequality
\begin{equation}\label{E:xnR-xnR'}
\begin{aligned}
1 - o_R(1) = \|\phi^R\|_{L^p(\lambda_\gamma)}^p & \notag =  \lim_n \int \big(\phi^{R} - e^{i (x_n^{R'} - x_n^{R}) \cdot \gamma} \phi^{R'}\big) \bar \phi^{R} \, |\phi^{R}|^{p - 2}
\\
& \lesssim \limsup_n \| e^{- i x_n^{R} \cdot \gamma} \phi^{R} - e^{- i x_n^{R'} \cdot \gamma} \phi^{R'}\|_{L^p(\lambda_\gamma)},
\end{aligned}
\end{equation}
contradicting~\eqref{e^{ix^Rcdot gamma} phi^R is cauchy}. Thus, $\{\|x_n^R - x_n^{R'}\|\}$ is bounded for all large $R, R' > R$. Replacing $\{f_n\}$ by the modulated sequence $\{e^{i x_n^{R_1} \cdot \gamma} f_n\}$ for some fixed $R_1$, and now passing to a subsequence, we may assume that for each fixed large $R$, each $x_n^R \to x^R \in \Rd$ as $n \to \infty$. Consequently, multiplication by $e^{i x_n^R \cdot \gamma}$ converges to multiplication by $e^{i x^R \cdot \gamma}$ in the strong operator topology. Combining this with \eqref{f_n close to profile in L^p} and the triangle inequality
$$
\limsup_n \|f_n - e^{- i x^R \cdot \gamma} \phi^R\|_{L^p(\lambda_\gamma)} = o_R(1).
$$
Applying the triangle inequality once more, this implies that
\begin{equation}\label{E : single profile for norm}
\lim_{R \to \infty} \sup_{R_1,R_2 \geq R}\|e^{- i x^{R_1} \cdot \gamma} \phi^{R_1} - e^{- i x^{R_2} \cdot \gamma} \phi^{R_2}\|_{L^p(\lambda_\gamma)} = 0.
\end{equation}
In summary, both $\{f_n\}$ and $\{e^{- i x^R \cdot \gamma} \phi^R\}$
being Cauchy sequences in $L^p(\lambda_\gamma)$ converge in $L^p(\lambda_\gamma)$ to the same $\phi$, an extremizer of $\scriptE_\gamma : L^p(\lambda_\gamma) \to L^q$. This completes the proof of Proposition~\ref{P:convergence} and of Theorem~\ref{T : existence} in the case of nonconcentration.  
\end{proof}


\section{$L^p$ convergence after pseudo-scaling in the case of concentration} \label{S:Lp concentration}


To complete the proof of Theorem~\ref{T : existence} it remains to prove that a normalized extremizing sequence $\{f_n\}$ obeying \eqref{E:conc freq loc} may be ``blown up'' to reveal either an extremizing sequence for extension from the moment curve $\gamma_0$, or a matched pair of such sequences, depending on the parity of the exponents $l_i$.  More precisely, we will prove the following proposition.  

\begin{proposition}\label{P:concentration}
If $\{f_n\}$ is a normalized extremizing sequence obeying \eqref{E:conc freq loc}, then there exists a subsequence along which
$$
\lim_{n \to \infty} \|e^{ix_n\cdot\gamma}S_n f_n - f_{\delta_n}\|_{L^p(\lambda_\gamma)} = 0,
$$
for some $\{x_n \in \Rd\}$. Here $f_{\delta_n}$ is as in \eqref{E:fdelta}, for extremizers $f, g$ of $\scriptE_{\gamma_0}$, obeying $|\scriptE_{\gamma_0} f| \equiv |\scriptE_{\gamma_0} g|$.  
\end{proposition}

The proof of Theorem~\ref{T : existence} will be complete once Proposition~\ref{P:concentration} is proved.  

The cases $\gamma$ even/odd and $\gamma$ neither even nor odd are sufficiently different from one another that we will take up each one separately, beginning with the simplest (neither).

\begin{proof}[Proof of Proposition~\ref{P:concentration} when $\gamma$ is neither even nor odd]
The argument is an adaptation of the proof of Proposition~\ref{P:convergence}, incorporating a blowup about $t=1$.  

We may assume that the dilation $S_n$ in \eqref{E:conc freq loc} equals the identity for all $n$. We define the truncation 
$$
f_n^R := f_n \rchi_{[1- R \delta_n, 1+R \delta_n] \cap \{|f_n| \leq R \delta_n^{-1/p}\}}
$$
and the blown up version, 
$$
F_n^R(t) := (\delta_n\lambda_\gamma(1))^{\frac 1p} f_n^R (1 + \delta_n t).  
$$
Recalling the notation introduced in \eqref{D : gamma(a, delta)} and \eqref{E : scriptE(a, delta) preserves L^q}, we see that 
\begin{equation} \label{E:FnR neither}
\begin{gathered}
\|F_n^R\|_{L^p(\lambda_{\gamma_{(1,\delta_n)}})} \equiv \|f_n^R\|_{L^p(\lambda_{\gamma})} = 1-o_R(1), \\
\|\scriptE_{\gamma_{(1,\delta_n)}}F_n^R\|_q \equiv \|\scriptE_\gamma f_n^R\|_q = B_{\gamma,p}-o_R(1),
\end{gathered}
\end{equation}
where the estimates on the right hand sides hold in the limit as $n \to \infty$.  

By \eqref{E : local approx}, we may apply Proposition~\ref{P : profiles} to $\{F_n^R\}$, with 
$$
\gamma_n:=\gamma_{(1, \delta_n)}, \qquad \gamma_L = \gamma_0.
$$
Recalling that $\lambda_{\gamma_0} \equiv 1$, (so $L^p(\lambda_{\gamma_0}) = L^p$), we let $\{\phi^{R,j}\}$ denote the $L^p$ functions arising in that decomposition.  By \eqref{E:FnR neither} and $B_{\gamma,p} \geq B_{\gamma_0,p}$ (by Theorem~\ref{T : conc energy lower bound}),
\begin{align*} 
B_{\gamma_0, p}^q - o_R(1) & = \liminf_n \|\scriptE_{\gamma_{(1,\delta_n)}}F_n^R\|_q^q 
\\
& \leq \sum_{j=1}^\infty \|\scriptE_{\gamma_0} \phi^{R,j}\|_q^q \leq B_{\gamma_0, p}^q \sum_{j = 1}^{\infty} \|\phi^{R,j}\|_p^q \leq B_{\gamma_0, p}^q \max_j \|\phi_j^R\|_p^{q - \tilde p}.
\end{align*}
By conclusion (iv) of Proposition~\ref{P : profiles} and \eqref{E:FnR neither}, for each sufficiently large $R$ there must exist a unique profile $\phi^{R,j}=:\phi^R \in L^p$ satisfying $\|\phi^R\|_p = 1 - o_R(1)$. 

By virtue of the uniform support condition, $\supp F_n^R \subseteq [-R,R]$ for all $n$, and \eqref{E : local approx}, the argument leading to \eqref{E : single profile for norm} adapts to the $F_n^R$ almost without any change.  For the convenience of the reader, we summarize that argument, updating the key estimates.  By strict convexity of $L^p$, 
\begin{equation} \label{E:FnR lim wk}
\limsup_n \|F_n^R - e^{-ix_n^R\cdot \gamma_n}\phi^R\|_p = o_R(1),
\end{equation}
so by \eqref{E:conc freq loc},
$$
\sup_{R_1,R_2 \geq R} \limsup_n \|e^{-ix_n^{R_1} \cdot \gamma_n}\phi^{R_1} - e^{-ix_n^{R_2}\cdot \gamma_n}\phi^{R_2}\|_p = o_R(1).
$$
By stationary phase and \eqref{E : local approx}, multiplication by $e^{i(x_n^{R_1}-x_n^{R_2})\cdot \gamma_n}$ tends weakly to zero if $|x_n^{R_1}-x_n^{R_2}| \to \infty$, and we obtain a contradiction analogously with \eqref{E:xnR-xnR'}.  Thus, after passing to a subsequence, $x_n^{R_1}-x_n^{R_2}$ converges for all sufficiently large $R_1,R_2$, so multiplication by $e^{i(x_n^{R_1}-x_n^{R_2})\cdot\gamma_n}$ converges in the strong operator topology on $L^p$.  By modifying $\phi^R$, we may thus upgrade \eqref{E:FnR lim wk} to 
$$
\lim_{n \to \infty} \|e^{ix_n\cdot \gamma_n}F_n^R - \phi^R\|_p = o_R(1).
$$
Having removed the dependence of the modulation on $R$, we may eliminate it (i.e., set $x_n \equiv 0$) by an initial modulation of $f_n$.  Moreover, by \eqref{E:conc freq loc}, $\{\phi^R\}$ is Cauchy in $R$, and hence as $R \to \infty$, converges in $L^p$ to some $f$ with $\|f\|_p = 1$. Undoing the pseudo-scaling by $\delta_n$ and using \eqref{E:conc freq loc} to return to use of $f_n$, 
$$
\lim_n \|f_n-f_{\delta_n}\|_{L^p(\lambda_\gamma)} = 0,
$$
with $f_{\delta_n}$ derived from $f$ as in \eqref{E:fdelta}. By \eqref{E:FnR neither}, we have that $\|\scriptE_{\gamma_0}f\|_q = B_{\gamma,p}$. Since $B_{\gamma_0,p} \leq B_{\gamma, p}$ (by Theorem~\ref{T : conc energy lower bound}), $f$ is an extremizer for $\scriptE_{\gamma_0}$ and $B_{\gamma_0,p} = B_{\gamma, p}$. This completes the proof of Proposition~\ref{P:concentration}, and thus of Theorem~\ref{T : existence} when $\gamma$ is neither odd nor even.
\end{proof}

We now turn to the cases when $\gamma$ is either even or odd, as the argument initially starts in the same way for both cases.  

\begin{proof}[Proof of Proposition~\ref{P:concentration} for $\gamma$ even/odd]
In the case that $\gamma$ is either even or odd, we may once again take the dilations $S_n$ in \eqref{E:conc freq loc} to equal the identity. We define 
\begin{align*}
f_n^R &:= f_n \rchi_{\{\dist(\cdot,[1])<R\delta_n, |f_n|<R\delta_n^{-1/p}\}}, \quad
F_n^{R,\pm}(s):=(\delta_n \lambda_\gamma(1))^{\frac 1p} (f_n^R \rchi_{I^{\pm}}) (\pm(1+\delta_ns)) 
\end{align*}
where $I^+ := (0, \infty)$ and $I^- := (- \infty, 0)$. Therefore,
\begin{equation} \label{E:fnR in terms of FnR}
f_n^R(t) = (\delta_n\lambda_\gamma(1))^{-1/p}\bigl[F_n^{R,+}\bigl(\tfrac{t-1}{\delta_n}\bigr) + F_n^{R,-}\bigl(\tfrac{-(t+1)}{\delta_n}\bigr)\bigr].  
\end{equation}
As in \eqref{E : scriptE(a, delta) preserves L^q} and further elementary manipulations,
\begin{equation}\label{E:EFnRpm}
\|\scriptE_\gamma f_n^R\|_q = \|\sum_{\nu \in \{+,-\}} e^{ix\cdot\xi_n^\nu}\scriptE_{\gamma_n^\nu}F_n^{R,\nu}(x)\|_{L^q_x},
\end{equation}
where
\begin{equation} \label{E:xin gamman}
\xi_n^{\pm}:=
\begin{cases}
0, &\text{$\gamma$ even;}\\
\pm D_{\delta_n^{-1}}T(1)^{-1}\gamma(1), &\text{$\gamma$ odd;}
\end{cases}
\qquad 
\gamma_n^{\pm}:=
\begin{cases}
\gamma_{(1,\delta_n)}, & \text{$\gamma$ even;}\\
\pm\gamma_{(1,\delta_n)}, &\text{$\gamma$ odd.}
\end{cases} 
\end{equation}
We note that $\gamma_n^\nu \to \gamma_0^\nu$ as $n \to \infty$, where
$$
\gamma_0^\nu:=\gamma_0, \:\text{if $\gamma$ is even; and} \: \gamma_0^\nu:=\nu\gamma_0, \:\text{if $\gamma$  is odd.}
$$ 
This leads us to consider the following generalization of Proposition~\ref{P : profiles}, which accommodates simultaneous concentration at antipodal points.  

\begin{proposition} \label{P:joint profile}
Let $q=\tfrac{d^2+d}2 p' > p$, $\gamma:\R \to \R^d$ an even or odd monomial, $R > 0$, and $\delta_n \searrow 0$.  Assume that the measurable sequences $\{F_n^\pm\}$ obey $|F_n^\pm| \leq R\rchi_{[-R,R]}$.  Then there exist $\{x_n^j \in \Rd\}$ and $\{\phi^{j,\pm}\in L^p\}$ such that along a subsequence, the following hold with 
$$
R_n^{J,\nu}:=F_n^\nu - \sum_{j=1}^J e^{-ix_n^j \cdot \gamma_n^\nu}\phi^{j,\nu}, \qquad \nu \in \{+,-\}.
$$
\begin{enumerate}[\rm(i')]
\item $\lim_{n \to \infty} \|x_n^j - x_n^{j'}\| = \infty$, $j \neq j'$;
\item  $e^{ix_n^j \cdot \gamma_n^\nu} F_n^\nu \rightharpoonup \phi^{j,\nu}$, weakly for all $j\in\N$, $\nu \in \{\pm\}$;
\item  $\lim_{n \to \infty} \bigl(\|\sum_{\nu \in \{\pm\}} e^{ix\cdot\xi_n^\nu}\scriptE_{\gamma_n^\nu}F_n^{\nu}(x)\|_{L^q_x}^q - \sum_{j=1}^J \|\sum_{\nu \in \{\pm\}}e^{ix\cdot\xi_n^\nu}\scriptE_{\gamma_0^\nu}\phi^{j,\nu}(x)\|_{L^q_x}^q- \|\sum_{\nu \in \{\pm\}} e^{ix\cdot\xi_n^\nu}\scriptE_{\gamma_n^\nu}R_n^{\nu}(x)\|_{L^q_x}^q\bigr) = 0$;
\item $\sum_{j=1}^\infty ( \|\phi^{j,+}\|_p^p+\|\phi^{j,-}\|_p^p)^{\tilde p/p} \leq \liminf (\|F_n^+\|_p^p+ \|F_n^-\|_p^p)^{\tilde p/p}$, \quad $\tilde p : = \max (p, p')$;
\item $\lim_{J \to \infty} \limsup_{n \to \infty} \|\scriptE_{\gamma_n}R_n^{J,\pm}\|_q = 0$.
\end{enumerate}
\end{proposition}

\begin{proof}[Proof of Proposition~\ref{P:joint profile}]
We apply Proposition~\ref{P : profiles} to each $F_n^{\pm}$ separately, obtaining profiles $\{\phi^{j,\pm}\}$ and sequences $\{x_n^{j,\pm}\}$ obeying the conclusions of that proposition, with the superscripts $\nu \in \{\pm\}$.  Conclusions (i), (ii), and (v) would immediately yield (i'), (ii'), and (v') if we had $x_n^{j,+}=x_n^{j,-}$ for all $n,j$, and deducing this (after a reordering and modification of the profiles) will be our first step.  

After passing to a subsequence, we may assume that for any pair $j,j'$, either $x_n^{j,+}-x_n^{j',-}$ converges or $|x_n^{j,+}-x_n^{j',-}|\to \infty$.  Given $j$, the case of convergence is only possible for at most one $j'$ (by conclusion (i)).  By reordering the profiles and inserting zeros, we may assume that $x_n^{j,+}-x_n^{j,-}$ converges for all $j$ and $|x_n^{j,\nu}-x_n^{j',\nu'}| \to \infty$ for all $j \neq j'$ and $\nu,\nu' \in \{+,-\}$.  Indeed, let $\scriptJ \subseteq \Z_{\geq 0}^2$ equal the set of all $(j,j')$ such that $x_n^{j,+}-x_n^{j',-}$ converges, all $(j,0)$ such that $|x_n^{j,+}-x_n^{j',-}| \to \infty$ for all $j'$, and all $(0,j')$ such that $|x_n^{j,+}-x_n^{j',-}| \to \infty$ for all $j$.  To $(j,j') \in \scriptJ$, we associate the pair $(\phi^{j,+},\phi^{j',-})$, where $\phi^{0,\pm}:=0$.  By conclusion (iii) of the original proposition, we may order the elements of $\scriptJ \ni (j_k,j_k')$ so that $\|\scriptE_{\gamma_0}\phi^{j_k,+}\|_q+\|\scriptE_{\gamma_0}\phi^{j_k,-}\|_q$ is decreasing in $k$.  Now we reindex (i.e., set $\tilde\phi^{k,\pm}:=\phi^{j_k,\pm}$ and then erase the tildes).  (See also \cite{Tautges}.)  Since $x_n^{j,+}-x_n^{j,-}$ converges, multiplication by $e^{-i(x_n^{j,-}-x_n^{j,+})\cdot \gamma_n^-}$ converges in the strong operator topology on $L^p$, so by modifying $\phi^{j,-}$ to absorb this limit, we may assume that $x_n^{j,+}=x_n^{j,-} =:x_n^j$ for all $n,j$, as desired.  It remains to establish conclusions (iii') and (iv').    

From the triangle inequality in $l^{\frac{\tilde p}{p}}$ (recall that $\tilde p \geq p$), then conclusion (iv) of Proposition~\ref{P : profiles}, 
\begin{align*}
\bigl(\sum_{j=1}^J \bigl(\|\phi^{j,+}\|_p^p+\|\phi^{j,-}\|_p^p\bigr)^{\frac{\tilde p}{p}}\bigr)^{\frac{p}{\tilde p}}
&\leq \bigl(\sum_{j=1}^J \|\phi^{j,+}\|_p^{\tilde p}\bigr)^{\frac{p}{\tilde p}} +\bigl(\sum_{j=1}^J \|\phi^{j,-}\|_p^{\tilde p}\bigr)^{\frac{p}{\tilde p}} 
\\
& \leq \liminf \|F_n^+\|_p^p+\|F_n^-\|_p^p,
\end{align*}
so conclusion (iv') holds.  

Finally, conclusion (iii') holds by precisely the same argument as that given for conclusion (iii) in Proposition~\ref{P : profiles}. For the convenience of the reader we record the updated version of \eqref{expand and apply Brezis-Lieb}, which is 
\begin{align*}
\sum_{\nu \in \{+,-\}} e^{ix\cdot\xi_n^\nu}\scriptE_{\gamma_n^\nu}F_n^{\nu}(x)
& =
\bigl[\sum_{\nu \in \{+,-\}} e^{ix\cdot\xi_n^\nu}\scriptE_{\gamma_0^\nu}\phi^{1,\nu}(x) \bigr]
\\
& + \bigl[\sum_{\nu \in \{+,-\}} e^{ix\cdot\xi_n^\nu}\scriptE_{\gamma_n^\nu}(e^{ix_n^1\cdot\gamma_n^\nu}F_n^\nu-\phi^{1,\nu})(x) \bigr]
\\
& + \bigl[\sum_{\nu \in \{+,-\}} e^{ix\cdot\xi_n^\nu}(\scriptE_{\gamma_n^\nu}\phi^{1,\nu}(x)-\scriptE_{\gamma_0^\nu}\phi^{1,\nu}(x)) \bigr].
\end{align*}
The first term in brackets on the right is bounded by a fixed $L^q$ function, the second converges to zero pointwise, and the third converges to zero in $L^q$ by Lemma~\ref{L:op convergence}. Applying the generalized Brezis-Lieb lemma to the above decomposition as in conclusion (iii) of Proposition~\ref{P : profiles} finishes the proof.   
\end{proof}

With Proposition~\ref{P:joint profile} proved, we return to the main thread of the proof of Proposition~\ref{P:concentration} in the even/odd case.

The lower bound in Theorem~\ref{T : conc energy lower bound} and our hypotheses that $\{f_n\}$ is $L^p(\lambda_\gamma)$-normalized, extremizing, and obeys \eqref{E:conc freq loc} and \eqref{E:EFnRpm} yield
\begin{equation} \label{E:EFnRpm lb}
(B_{\gamma,p}^{\rm{conc}})^q - o_R(1) = \liminf_{n \to \infty} \|\sum_{\nu \in \{+,-\}} e^{ix\cdot\xi_n^\nu}\scriptE_{\gamma_n^\nu}F_n^{R,\nu}(x)\|_{L^q_x}^q.  
\end{equation}
Applying Proposition~\ref{P:joint profile} to the $F_n^{R,\pm}$, we obtain profiles $\{\phi^{j,R,\pm} \in L^p\}$ and points $\{x_n^{j,R} \in \R^d\}$ which, by virtue of \eqref{E:EFnRpm lb}, conclusions (iii'), and (v') obey
\begin{equation} \label{E:Ephipm lb}
(B_{\gamma,p}^{\rm{conc}})^q - o_R(1)
\leq
\sum_{j=1}^\infty \liminf_{n \to \infty} \|\sum_{\nu \in \{+,-\}}e^{ix\cdot\xi_n^\nu}\scriptE_{\gamma_0^\nu}\phi^{j,R,\nu}(x)\|_{L^q_x}^q.
\end{equation}

Separately considering the cases $\gamma$ even and $\gamma$ odd leads to a dramatic simplification of \eqref{E:Ephipm lb}, and our argument now bifurcates.  

When $\gamma$ is even, \eqref{E:Ephipm lb} states that 
\begin{equation} \label{E:Ephipm lb e}
2^{\frac{q}{p'}}B_{\gamma_0,p}^q - o_R(1) \leq \sum_{j=1}^\infty\|\scriptE_{\gamma_0}(\phi^{j,R,+}+\phi^{j,R,-})\|_q^q.
\end{equation}
We may use continuity of $\scriptE_{\gamma_0}$, the triangle inequality, and H\"older's inequality to estimate a single summand on the right hand side of \eqref{E:Ephipm lb e}:
\begin{equation} \label{E:phiR extremizing e}
\begin{aligned}
&\|\scriptE_{\gamma_0}(\phi^{j,R,+}+\phi^{j,R,-})\|_q 
\leq B_{\gamma_0,p}\|\phi^{j,R,+}+\phi^{j,R,-}\|_p\\
&\qquad \leq B_{\gamma_0,p}\bigl(\|\phi^{j,R,+}\|_p + \|\phi^{j,R,-}\|_p\bigr)
\leq B_{\gamma_0,p}2^{\frac1{p'}}\bigl(\|\phi^{j,R,+}\|_p^p + \|\phi^{j,R,-}\|_p^p\bigr)^{\frac1p}.
\end{aligned}
\end{equation}
Inserting this upper bound into \eqref{E:Ephipm lb e} and using H\"older's inequality and $q>\tilde p$, then conclusion (iv') and the hypothesis that $f_n$ is $L^p(\lambda_\gamma)$-normalized,
\begin{align*}
&2^{\frac{q}{p'}}B_{\gamma_0,p}^q - o_R(1)
\leq 2^{\frac{q}{p'}}B_{\gamma_0,p}^q \sum_j \bigl(\|\phi^{j,R,+}\|_p^p + \|\phi^{j,R,-}\|_p^p\bigr)^{\frac qp}\\
&\qquad \leq 2^{\frac{q}{p'}}B_{\gamma_0,p}^q \sup_j \bigl(\|\phi^{j,R,+}\|_p^p + \|\phi^{j,R,-}\|_p^p\bigr)^{\frac{q-\tilde p}p} \limsup_{n \to \infty} (\|F_n^{R,+}\|_p^p+ \|F_n^{R,-}\|_p^p)^{\tilde p/p}\\
&\qquad\leq 2^{\frac{q}{p'}}B_{\gamma_0,p}^q \sup_j \bigl(\|\phi^{j,R,+}\|_p^p + \|\phi^{j,R,-}\|_p^p\bigr)^{\frac{q-\tilde p}p}.  
\end{align*}
Thus, for each $R$, there exists a large profile (whose superscript $j$ we omit) such that
$$
1-o_R(1) \leq \|\phi^{R,+}\|_p^p + \|\phi^{R,-}\|_p^p \leq \liminf \|F_n^{R,+}\|_p^p+ \|F_n^{R,-}\|_p^p \leq 1.
$$
Thus, 
$$
\liminf \|F_n^{R,\nu}\|_p - o_R(1) \leq  \|\phi^{R,\nu}\|_p \leq \liminf \|F_n^{R,\nu}\|_p, \qquad \nu \in\{\pm\},
$$
and so by exactly the same argument as in the ``neither'' case, we obtain a single ($R$ independent) sequence $\{x_n \in \Rd\}$ such that
\begin{equation} \label{E:FnR phinR e}
\lim_{n \to \infty} \|e^{ix_n\cdot\gamma_n^\nu}F_n^{R,\nu} - \phi^{R,\nu}\|_p = o_R(1), \quad \text{as $R \to \infty$.}
\end{equation}
By modulating the original sequence $\{f_n\}$, we may assume that $x_n =0$ for all $n$.  
As in the ``neither'' case, we may use \eqref{E:conc freq loc} and the triangle inequality to show that there exist $\phi^\pm \in L^p$ such that $\phi^{R,\nu} \to \phi^\nu$ in $L^p$ as $R \to \infty$, for $\nu \in \{\pm\}$.  Revisiting \eqref{E:Ephipm lb e} and \eqref{E:phiR extremizing e}, equality must hold at each step of the chain of inequalities
\begin{align*}
&2^{\frac{1}{p'}}B_{\gamma_0,p}^q \leq \|\scriptE_{\gamma_0}(\phi^{+}+\phi^{-})\|_q 
\leq B_{\gamma_0,p}\|\phi^{+}+\phi^{-}\|_p\\
&\qquad \leq B_{\gamma_0,p}\bigl(\|\phi^{+}\|_p + \|\phi^{-}\|_p\bigr)
\leq B_{\gamma_0,p}2^{\frac1{p'}}\bigl(\|\phi^{+}\|_p^p + \|\phi^{-}\|_p^p\bigr)^{\frac1p} = 2^{\frac1{p'}}B_{\gamma_0,p}.
\end{align*}
Taking the inequalities in order, $\phi^++\phi^-$ is an extremizer for $\scriptE_{\gamma_0}$, $\phi^-$ is a constant nonnegative multiple of $\phi^+$ (or vice versa), and $\|\phi^-\|_p=\|\phi^+\|_p$.  Thus $\phi^+ = \phi^- =:\phi$, an extremizer of $\scriptE_{\gamma_0}$.  We may now upgrade \eqref{E:FnR phinR e} to 
$$
\lim_{R \to \infty} \lim_{n \to \infty} \|F_n^{R,\nu}-\phi^\nu\|_p = 0,
$$
where recall the reduction (stated below \eqref{E:FnR phinR e}) to $x_n \equiv 0$.  Returning to \eqref{E:conc freq loc}, and looking to \eqref{E:fnR in terms of FnR} for assistance, we see that after rescaling, we do indeed have the expected $L^p$ convergence $\lim_{n \to \infty} \|f_n-f_{\delta_n}\|_p=0$, with $f:=\phi$ and $f_\delta$ as in \eqref{E:fdelta}.  

The case $\gamma$ odd concludes analogously, as we will now see. In this case, recalling \eqref{E : define Psi_p,q},  \eqref{E:Ephipm lb} may be restated as (observe that $\scriptE_{- \gamma_0} g = \overline{\scriptE_{\gamma_0}\overline{g}}$)
\begin{equation} \label{E:Ephipm lb o}
\Psi_{p,q}B_{\gamma_0,p}^q - o_R(1) \leq \sum_{j=1}^\infty \liminf_{n \to \infty} \|e^{2ix\cdot \xi_n^+}\scriptE_{\gamma_0}\phi^{j,R,+} + \overline{\scriptE_{\gamma_0}\overline{\phi^{j,R,-}}}\|_q^q.  
\end{equation}
Recalling Lemma~\ref{L : limiting constant} and expanding $x \cdot \xi_n^+ = (D_{\delta_n^{-1}} x)\cdot T(1)^{-1}\gamma(1)$, we may estimate a summand from the right hand side of \eqref{E:Ephipm lb o}:
$$
\liminf_{n \to \infty} \|e^{2ix\cdot \xi_n^+}\scriptE_{\gamma_0}\phi^{j,R,+} + \overline{\scriptE_{\gamma_0}\overline{\phi^{j,R,-}}}\|_q^q \leq \Psi_{p,q}B_{\gamma_0,p}^q (\|\phi^{j,R,+}\|_p^p + \|\phi^{j,R,-}\|_p^p)^{\frac qp},
$$
and arguing precisely as in the ``even'' case, we use $q>\tilde p$ to deduce existence of a large pair $\phi^{R,\pm}$, for each $R$.  Moreover, each $\{\phi^{R,\pm}\}$ is Cauchy and hence $L^p$ convergent as $R\to\infty$, with limit $\phi^\pm$.   We now have equality at each stage of the inequalities
\begin{align*}
&(\Psi_{p,q})B_{\gamma_0,p}^q = \liminf_{n \to \infty} \|e^{2ix\cdot \xi_n^+}\scriptE_{\gamma_0}\phi^{+} + \overline{\scriptE_{\gamma_0}\overline{\phi^{-}}}\|_q^q  \\
&\qquad
 \leq (\Psi_{p,q}) (\|\scriptE_{\gamma_0}\phi^{+}\|_q^p + \|\scriptE_{\gamma_0}\phi^{-}\|_q^p)^{\frac qp} 
\leq (\Psi_{p,q})B_{\gamma_0,p}^q(\|\phi^{+}\|_p^p + \|\phi^{-}\|_p^p)^{\frac qp}.  
\end{align*}
Considering the case of equality for each inequality in turn, (recalling Lemma~\ref{L : limiting constant}) $|\scriptE_{\gamma_0}\phi^+| = \alpha |\scriptE_{\gamma_0}\overline{\phi^-}|$ a.e.\ (or vice versa), for $\alpha \in [0,1]$ a maximum point of $\psi_{p,q}$, with $\phi^\pm$ extremal for $\scriptE_{\gamma_0}$.  Letting $f:=\phi^+$ and $g:=\overline{\phi^-}$ yields $\lim_{n \to \infty} \|f_n-f_{\delta_n}\|_p = 0$ (after modulation).  
\end{proof}
 




\bibliographystyle{plain}
\bibliography{res_monomial_curve}


\end{document}